\documentclass[reqno,a4paper,10pt]{article}
\usepackage{amsmath,amssymb,amsthm,epsfig,float,color,mathtools,enumitem,mathrsfs,fancyhdr}

\usepackage[top=2.75cm, bottom=2.5cm, left=2.5cm, right=2.5cm]{geometry}
\definecolor{lightblue}{rgb}{0.22,0.45,0.70}
\definecolor{mygray}{rgb}{0.7,0.7,0.7}
\definecolor{green1}{rgb}{0.09, 0.45, 0.27}
\definecolor{lilac}{rgb}{0.784, 0.635, 0.784}

\usepackage[colorlinks=true,breaklinks=true,linkcolor=lightblue,citecolor=lightblue]{hyperref}

\newtheorem{theorem}{Theorem}[section]
\newtheorem{remark}{Remark}[section]
\newtheorem{lemma}{Lemma}[section]

\numberwithin{equation}{section}
\numberwithin{theorem}{section}
\numberwithin{figure}{section}
\numberwithin{table}{section}
\numberwithin{remark}{section}
\numberwithin{lemma}{section}
\numberwithin{corollary}{section}

\definecolor{lightblue}{rgb}{0.22,0.45,0.70}
\definecolor{lightgreen}{rgb}{0.22,0.50,0.25}

\newcommand\cero{\boldsymbol{0}}
\newcommand{\norm}[1]{\left\|#1\right\|}

\newcommand\vdiv{\mathop{\mathrm{div}}\nolimits}
\newcommand\curl{\mathop{\mathrm{curl}}\nolimits}
\newcommand\bcurl{\mathop{\mathbf{curl}}\nolimits}

\newcommand\dist{\mathop{\mathrm{dist}}\nolimits}

\newcommand{\bS}{\mathbf{S}}

\newcommand{\bI}{\mathbf{I}}

\newcommand{\bu}{\boldsymbol{u}}

\newcommand{\bg}{\boldsymbol{g}}

\newcommand{\bn}{\boldsymbol{n}}

\newcommand{\bv}{\boldsymbol{v}}
\newcommand{\bw}{\boldsymbol{w}}
\newcommand{\bx}{\boldsymbol{x}}

\newcommand{\bz}{\boldsymbol{z}}
\newcommand{\br}{\mathbf{r}}
\newcommand{\be}{\boldsymbol{e}}
\newcommand{\ff}{\boldsymbol{f}}
\renewcommand{\gg}{\boldsymbol{g}}

\newcommand{\bnabla}{\boldsymbol{\nabla}}
\newcommand{\btau}{\boldsymbol{\tau}}
\newcommand{\bchi}{\boldsymbol{\chi}}

\newcommand\cA{\mathcal{A}}

\newcommand\cD{\mathcal{D}}

\newcommand\cT{\mathcal{T}}
\newcommand\cE{\mathcal{E}}
\newcommand\cF{\mathcal{F}}

\newcommand\cS{\mathcal{S}}
\newcommand\cP{\mathcal{P}}

\newcommand\bB{\mathbf{B}}

\newcommand{\ds}{\displaystyle}
\def\qan{{\quad\hbox{and}\quad}}
\def\wt{\widetilde}

\newcommand\bV{\mathbf{V}}

\newcommand\bH{\mathbf{H}}
\newcommand\bL{\mathbf{L}}
\newcommand\rL{\mathrm{L}}

\newcommand\bW{\mathbf{W}}
\newcommand\bZ{\mathbf{Z}}

\newcommand\rH{\mathrm{H}}
\newcommand\rQ{\mathrm{Q}}
\newcommand\rY{\mathrm{Y}}

\newcommand\bgamma{\boldsymbol{\gamma}}

\newcommand\bDelta{\boldsymbol{\Delta}}

\newcommand\bomega{\boldsymbol{\omega}}

\newcommand\bzeta{\boldsymbol{\zeta}}

\renewenvironment{proof}{\noindent{\it Proof.}}{\hfill$\square$}

\allowdisplaybreaks
\title{Numerical analysis of a porous natural convection system with vorticity and viscous dissipation\thanks{\textbf{Updated:} \today. \textbf{Funding:} This work has been partially supported by the Monash Mathematics Research Fund S05802-3951284; by the Australian Research Council through the \textsc{Future Fellowship} grant FT220100496 and \textsc{Discovery Project} grant DP22010316; and by the National Research and Development Agency (ANID) of the Ministry of Science, Technology, Knowledge and Innovation of Chile through the postdoctoral program \textsc{Becas Chile} grant 74220026.}}

\author{
Russel Demos\thanks{School of Mathematics, Monash University, 9 Rainforest Walk, Melbourne 3800 VIC, Australia. Email: \texttt{russel.demos@monash.edu}.}
\and 
Rashmi Dubey\thanks{School of Engineering, University of Petroleum and Energy Studies, Bidholi Campus, Dehradun, Uttarakhand 248007, India. Email: \texttt{rashmi.dubey@ddn.upes.ac.in}.}
\and 
Ricardo Ruiz-Baier\thanks{(Corresponding author). School of Mathematics, Monash University, 9 Rainforest Walk, Melbourne 3800 VIC, Australia; and Universidad Adventista de Chile, Casilla 7-D Chill\'an, Chile. Email: \texttt{ricardo.ruizbaier@monash.edu}.} 
\and 
Segundo Villa-Fuentes\thanks{School of Mathematics, Monash University, 9 Rainforest Walk, Melbourne 3800 VIC, Australia. Email: \texttt{segundo.villafuentes@monash.edu}.}
}

\date{}
\begin{document}
\maketitle

\begin{abstract}
\noindent 
In this paper we propose and analyse a new formulation and  pointwise divergence-free mixed finite element methods for the numerical approximation of Darcy--Brinkman equations in vorticity--velocity--pressure form, coupled with a transport equation for thermal energy with viscous dissipative effect and mixed Navier-type boundary conditions. The solvability analysis of the continuous and discrete problems is significantly more involved than usual as it hinges on Banach spaces needed to properly control the advective and dissipative terms in the non-isothermal energy balance equation. We proceed by decoupling the set of equations and use the Banach fixed-point theorem in combination with the abstract theory for perturbed saddle-point problems. Some of the necessary estimates are straightforward modifications of well-known results, while other technical tools require a more elaborated analysis. The velocity is approximated by Raviart--Thomas elements, the vorticity uses N\'ed\'elec spaces of the first kind, the pressure is approximated by piecewise polynomials, and the temperature by continuous and piecewise polynomials of one degree higher than pressure. Special care is needed to establish discrete inf-sup conditions since the curl of the discrete vorticity is not necessarily contained in the discrete velocity space, therefore suggesting to use two different Raviart--Thomas interpolants. A discrete fixed-point argument is used to show well-posedness of the Galerkin scheme. Error estimates in appropriate norms are derived, and a few representative numerical examples in 2D and 3D and with mixed boundary conditions are provided. 
\end{abstract}



\maketitle

\section{Introduction}
\paragraph{Scope.} The interplay of vorticity and viscous dissipation in natural convection within high permeability porous media results in intricate flow and thermal dynamics. Vorticity enhances mixing and improves heat transfer efficiency, while viscous dissipation raises the local temperature, potentially altering the overall thermal gradient that drives convection. To accurately predict the system performance and design effective solutions, these interactions need to be thoroughly understood and appropriately modelled. We consider the numerical analysis of the coupled Brinkman equations with temperature including a viscous dissipation term. In fluid saturated porous domains with relatively large permeabilities (as in the regime where Brinkman equations hold), viscous dissipation effects within the fluid are typically non-negligible as discussed in \cite{dubey19b}. In the case of clear fluid, the viscous dissipation term typically includes a nonlinearity in the form of the velocity scaled by its own Laplacian. In contrast, the type of viscous dissipation used here (and more appropriate for flow in porous media) has a nonlinear term depending on the square filtration velocity modulus and it represents the power per unit volume generated by the dissipation (see also \cite{barletta21}). Some of the common features of viscous dissipation (acting as an internal heat generation mechanism) include the onset of unstable temperature gradients and secondary buoyant flow patterns that are not induced by external thermal forcing terms. 

The formulation of incompressible viscous flow equations using vorticity (or  microrotation), velocity and pressure has been used and analysed extensively in, e.g.,  \cite{amara07,amara04,anaya16,anaya19,bernardi06,chang90,duan03,dubois03,caraballo24,hanot23,karlsen11,olshanskii15,salaun15,tsai05}. Later on, models for coupled advection-diffusion equations and vorticity-based viscous flow formulations have been introduced in \cite{anaya18,lenarda17} (see also the recent contribution \cite{gao24}).  For the analysis in this case one needs to require higher  regularity of the Brinkman filtration velocity, for example. One can also follow \cite{bernardi18} (designed for the coupling of thermal energy equation and Darcy flow in mixed form) treating the advective term with a duality argument and invoking continuity and compactness of suitably chosen operators, and requiring further that the concentration-dependent permeability is Lipschitz continuous and uniformly bounded away from zero. Alternatively to those approaches, here we follow the series of works  \cite{girault90,caucao23,correa22,gatica22} (see also the references therein) where the functional structure of the problem is generalised to Banach spaces. This strategy has a number of advantages, such as relaxing the regularity assumptions and  not needing to include augmented Lagrangian terms, but it exhibits a more involved functional structure. This approach requires to find suitable Banach spaces that would allow us to show boundedness of all linear and nonlinear forms. We herein work with vorticity in the space $\bH(\bcurl_s,\Omega)$ (vector fields in $\bL^2(\Omega)$ whose curl is in $\bL^s(\Omega)$) and velocity in the space $\bH^r(\vdiv,\Omega)$ (vector fields in $\bL^r(\Omega)$ whose divergence is in $\rL^2(\Omega)$), with $\frac1r+\frac1s = 1$. 

We use a fixed-point approach to separate the Brinkman and viscous dissipated heat equations, and the  functional framework for the Brinkman equations requires the analysis of a perturbed saddle-point problem. The required inf-sup conditions necessitate non-standard regularity properties of auxiliary boundary value problems not typically available in the literature.  For vector potentials it is possible to use the elegant theory from the series of papers  \cite{amrouche98,amrouche21,amrouche13}  (see also \cite{poirier20}), which holds for any Lebesque exponent larger than one. However, the assumptions that lead to that theory do not apply to our case since the two sub-boundaries in our domain meet. Therefore we appeal to the works in \cite{mazya07,mazya09} and use a new auxiliary problem in the range of admissible Lebesgue exponents. 

Another distinctive feature of the present paper appears at the discrete level. As a consequence again of the functional structure, one of the issues is that the curl of the discrete vorticity is not in the same space as the discrete velocity. We then need to define two Raviart--Thomas type interpolators and use commutativity properties in the $\bH(\bcurl)$ space associated with vorticity and the extra Raviart--Thomas space. The approach used here might be of benefit for the analysis of other systems with similar Banach structure.

\paragraph{Plan of the paper.} The rest of the manuscript has been organised as follows. Notations and basic definitions to be utilised throughout the paper are collected in the remainder of this section. Section \ref{sec:model} states the strong form of the coupled problem in its classical form and also in terms of vorticity. There we also give a weak formulation. 
The well-posedness analysis of the continuous weak problem is developed in Section~\ref{sec:wellp} using a fixed-point approach. A finite element  method is defined in Section~\ref{sec:FE}, where we also derive the analysis of existence and uniqueness of solution to the discrete problem.  In Section~\ref{sec:error} we recall interpolation properties of specifies finite element subspaces, and derive a general C\'ea estimate. 
Section \ref{sec:results} is devoted to showing numerically the optimal convergence of the method and we perform some benchmark tests to further validate the proposed computational methods. 

\paragraph{Preliminaries and recurrent notation.} 
We will adopt standard terminology for Sobolev spaces and norms. Throughout the text, given a normed space $S$, by the boldface symbol $\bS$ we will denote the vector  extension $S^d$. If $S$ is a generic Banach space we denote its dual space by $S'$. The continuous and surjective linear trace map is denoted as $\gamma:\rH^1(\Omega)\to \rH^{1/2}(\partial \Omega)$ for which when $u$ is continuous then $\gamma(u) = u|_{\partial\Omega}$. The boldface $\bgamma$ will denote the vector-valued counterpart of $\gamma$. Similarly, the normal and tangential trace operators are denoted by 
$\gamma_{\bn}:\bH(\vdiv,\Omega)\to \rH^{-1/2}(\partial\Omega)$ and $\bgamma_{\btau}:\bH(\bcurl,\Omega)\to \bH^{-1/2}(\partial\Omega)$, respectively.  We employ $\cero$ to denote a generic null vector, and follow the convention that $C$, with or without subscripts, denotes a generic positive constant independent of the discretisation parameters, which may take different values at different instances. We use the notation $A \lesssim B$ for the inequality $A \leq CB$, where $A$ and $B$ are two scalar fields, and $C$ is a generic constant.

\section{Model problem and its weak formulation}\label{sec:model}
\subsection{Natural convection with viscous dissipation}
Let us consider a simply connected bounded and Lipschitz domain $\Omega\subset\mathbb{R}^d$, with $d \in \{2,3\}$ occupied by an incompressible fluid in a non-isothermal regime and moving within a fully saturated porous media. 
The domain boundary $\partial\Omega$ 
is partitioned into disjoint sub-boundaries where slip and tangential vorticity trace conditions are imposed  $\partial\Omega:= \overline{\Gamma} \cup \overline{\Sigma}$, $\overline{\Gamma}\cap \overline{\Sigma} = \emptyset$, and it is assumed for sake of simplicity that both sub-boundaries are non-empty $|\Gamma|\cdot|\Sigma|>0$. 
For a sufficiently smooth external body force $\ff:\Omega\to\mathbb{R}^d$ and external heat source $g:\Omega \to \mathbb{R}$, we consider the following form of the equations of steady natural convection in their  velocity--pressure--temperature form and using dimensional units (see, for example, \cite{dubey19} but here not taking into account the Rayleigh dissipation)  
\begin{subequations}\label{eq:coupled}
\begin{align}
\frac{\mu}{\kappa}\bu -\mu' \bDelta \bu + \nabla p & = \rho \ff(T) & \quad \text{in $\Omega$}, \\
\vdiv\bu & =0  & \quad \text{in $\Omega$}, \\
\sigma_0  T  + \bu\cdot\nabla T - \alpha \Delta T & = g +  \frac{\mu}{\kappa c'\rho} \bu\cdot\bu 
& \quad \text{in $\Omega$},
\end{align}
stating the balance of linear momentum and of mass, and the balance of energy including the viscous dissipation  function on the right-hand side, which is proportional to the square modulus of the seepage velocity (the kinetic energy, see \cite{rees17}).  
While viscous dissipation is a significant factor in driving convection inside the fluid-saturated porous medium, fluid phase and solid phase of the porous medium are assumed to be in local thermal equilibrium. In high-permeability porous media, the flow behaviour becomes more complex due to the intensified interaction between the fluid and the porous matrix. Such an interaction is significantly affected by, e.g., vorticity patterns and viscous dissipation, which are crucial in shaping the convection patterns and thermal distribution within the medium. 

We also consider the following set of boundary conditions representing zero tangential traces of vorticity and slip velocity (the so-called Navier boundary condition) together with insulated boundaries on the component $\Gamma$; while non-tangential flow velocity, vanishing  pressure, and fixed temperature are prescribed on the remainder of the boundary: 
\begin{align}
\bcurl \bu\times \bn  = \cero, \quad \bu\cdot \bn  = 0  \quad \text{and} \quad  \alpha\nabla T \cdot\bn = 0& \quad \text{on $\Gamma$},\label{bc:Gamma}\\
\bu \times \bn = \cero, \quad p = 0 \quad \text{and} \quad T = 0  & \quad \text{on $\Sigma$}\label{bc:Sigma}.
\end{align}
\end{subequations}
The model coefficients are the  fluid mass density $\rho$, the dynamic  viscosity $\mu$, the effective dynamic viscosity in the porous layer  $\mu'$,  the permeability of the porous medium $\kappa$, specific heat capacity per unit mass of the fluid  $c'$, heat capacity ratio $\sigma_0$, and the thermal diffusivity $\alpha$. The  linear Oberbeck--Boussinesq approximation is considered, and the exerted force due to changes in temperature is linearly temperature-dependent 
  \[\ff(T) = -\beta(T - T_0) \bg,\]
  with $\beta$ a positive constant (thermal expansion coefficient), $T_0$ a reference temperature, and $\bg$ gravity. 

Problem  \eqref{eq:coupled} can be equivalently set in terms of  velocity, vorticity, pressure, and temperature. For this we rewrite the momentum and energy balance equations using the rescaled vorticity vector
\[ \bomega := \sqrt{\mu'} \bcurl \bu,\]
and then use the vector identity 
\begin{equation*}
\bcurl \bcurl \bu = -\boldsymbol{\Delta}\bu + \nabla(\vdiv \bu),\end{equation*}
together with the incompressibility constraint and the fact that the apparent viscosity is assumed constant, which leads to the following set of equations 
\begin{subequations}\label{eq:ns-new}
\begin{align}
\bomega - \sqrt{\mu'} \bcurl \bu& = \cero & \quad \text{in $\Omega$}, \\
\frac{\mu}{\kappa}\bu+ \sqrt{\mu'} \bcurl\bomega +   \nabla p &  =\rho \ff(T)& \quad \text{in $\Omega$}, \\ 
\vdiv\bu & =0  & \quad \text{in $\Omega$}, \\
\sigma_0  T  + \bu\cdot\nabla T - \alpha \Delta T & = g+ \frac{\mu}{\kappa c'\rho} |\bu|^2  
& \quad \text{in $\Omega$},
\end{align}\end{subequations}
and the compatible vorticity trace boundary condition (the first relation in \eqref{bc:Gamma}) reads now $\bomega \times \bn = \cero$ on $\Gamma$.

\subsection{Weak formulation}
We start by considering the real numbers $r>1,s>1$ and we recall the definition of the following functional spaces 
\begin{gather*}
 \bH(\bcurl_{s},\Omega)=\left\{\bzeta \in \bL^2(\Omega): \bcurl\bzeta \in \bL^{s}(\Omega) \right\}, \qquad   \bH^r(\vdiv,\Omega)=\left\{\bv \in \bL^r(\Omega): \vdiv\bv \in \rL^2(\Omega) \right\}, \\
 \rH^1(\Omega) = \{\psi\in \rL^2(\Omega): \nabla\psi \in \bL^2(\Omega)\}, 
  \end{gather*}
equipped with the following  norms 
\begin{gather}\nonumber  \norm{\bzeta}_{\bcurl_s,\Omega}:=\norm{\bzeta}_{0,\Omega}+\norm{\bcurl\bzeta}_{\bL^s(\Omega)},\qquad \norm{\bv}_{r,\vdiv,\Omega}:=\norm{\bv}_{\bL^r(\Omega)}+\norm{\vdiv\bv}_{0,\Omega},\\
\norm{S}^2_{1,\Omega}:=\norm{S}^2_{0,\Omega}+\norm{\nabla S}^2_{0,\Omega},
   \label{eq:norms} \end{gather}
respectively. The first two spaces are Banach and the third space is Hilbert with respect to the  respective graph topology. 
In view of the boundary conditions, we also define the  following closed subspaces 
\begin{align*}
 \bH_\star(\bcurl_s,\Omega)& := \{\bzeta\in \bH(\bcurl_s,\Omega): \bgamma_{\btau}\bzeta = \cero \text{ on } \Gamma\},\nonumber\\
  \bH^r_\star(\vdiv,\Omega) & := \{\bv\in \bH^r(\vdiv,\Omega): \gamma_{\bn}\bv = 0 \text{ on } \Gamma\}, \\
  \rH^1_\star(\Omega) &:=\{ S \in \rH^1(\Omega): \gamma S = 0 \text{ on } \Sigma\}, \nonumber
   \end{align*}
 where the boundary specification is to be understood in the sense of traces restricted to sub-boundaries.  
Note that if $\bzeta \in \bH(\bcurl_s,\Omega)$ then its tangential component $\bgamma_{\btau}\bzeta$ is defined in the trace space $\bW^{-1/s,s}(\partial\Omega)$ and if $\bv \in \bH^r(\vdiv,\Omega)$ then its normal component  $\gamma_{\bn}\bv$ is  defined in $\mathrm{W}^{-1/r,r}(\partial\Omega)$. 

We proceed to multiply the momentum balance, constitutive, mass, and energy equations by suitable test functions and to integrate by parts over the domain. Note that for the divergence-based terms we use the following form of the Gauss formula conveniently extended to the case of Banach spaces (see, e.g., \cite[Section 2.2]{caucao20b}): 
\[ \int_\Omega \bv\cdot \nabla q = - \int_\Omega \vdiv \bv\, q + \langle \gamma_{\bn}\bv, \gamma q \rangle \qquad \forall \bv\in\bH^t(\vdiv,\Omega),q \in \mathrm{W}^{1,t'}(\Omega),\]
where $\langle \cdot,\cdot\rangle$ denotes the duality product between $\mathrm{W}^{-1/t,t}(\partial\Omega)$ and  $\mathrm{W}^{1/t,t'}(\partial\Omega)$, with $1<t<\infty$. Similarly,  
 for curl-based terms we use a   generalisation of  \cite[Theorem 2.11]{girault79} (see, e.g., \cite[Section 2]{amrouche21}):  
\[ \int_\Omega \bcurl\bv\cdot \bzeta = \int_\Omega \bv\cdot \bcurl\bzeta + \langle \bgamma_{\btau}\bzeta, \bgamma\bv\rangle \qquad \forall \bzeta\in\bH(\bcurl_t,\Omega),\bv\in \bW^{1,t'}(\Omega),\]
where $\langle \cdot,\cdot\rangle$ denotes the pairing between $\bW^{-1/t,t}(\partial\Omega)$ and its dual $\bW^{1/t,t'}(\partial\Omega)$. 

We then arrive at the following weak formulation for \eqref{eq:ns-new}: for given $\bg\in \bL^\infty(\Omega)$ and $g\in L^2(\Omega)$,  find $(\bomega,\bu,p,T)\in \bH_\star(\bcurl_s,\Omega)\times  \bH^r_\star(\vdiv,\Omega)\times \rL^2(\Omega)  \times \rH_\star^1(\Omega)$ such that 
\begin{subequations}\label{eq:weak} 
\begin{align}
\int_\Omega \bomega \cdot\bzeta  - \sqrt{\mu'}\int_\Omega \bu \cdot \bcurl \bzeta 
& = 0\qquad  \forall \bzeta\in {\bH_\star(\bcurl_s,\Omega)} ,\\
- \sqrt{\mu'} \int_\Omega \bcurl\bomega \cdot \bv - \frac{\mu}{\kappa} \int_\Omega  \bu \cdot \bv  + \int_\Omega p \vdiv \bv & = 
- \rho \int_\Omega \ff(T)\cdot\bv \qquad  \forall \bv\in{\bH^r_\star(\vdiv,\Omega)}, \\
\int_\Omega q\vdiv\bu &= 0\qquad  \forall q\in {\rL^2(\Omega)}, \\
\sigma_0 \int_\Omega T\,S + \int_\Omega (\bu \cdot \nabla T)\,S  
+\alpha \int_\Omega \nabla T \cdot \nabla S    
& = \int_\Omega g \,S +   \frac{\mu}{\kappa c'\rho} \int_\Omega |\bu|^2 \,S 
\qquad \forall  S \in {\rH^1_\star(\Omega)},
\end{align}\end{subequations}
where we have also used the boundary conditions \eqref{bc:Gamma}--\eqref{bc:Sigma}. 

For $\bomega,\bzeta\in \bH_\star(\bcurl_s,\Omega)$, $\bu,\bv \in \bH^r_\star(\vdiv,\Omega)$, $p,q\in\rL^2(\Omega)$,  $T,S \in \rH^1_\star(\Omega)$, we define the following bilinear and trilinear forms, as well as the (affine,linear)-form $F$ and linear functional $G$:
\begin{gather*}
a_1(\bomega,\bzeta) : = \int_\Omega \bomega \cdot\bzeta,
 \qquad
b_1(\bzeta,\bv) : =   - \sqrt{\mu'} \int_\Omega \bcurl\bzeta \cdot \bv, \qquad
b_2(\bv,q) : =  \int_\Omega q\vdiv \bv, \\
a_2 (\bu,\bv) := \frac{\mu}{\kappa} \int_\Omega  \bu \cdot \bv, \quad 
a_3(T,S):=\sigma_0 \int_\Omega T\,S + \alpha \int_\Omega \nabla T \cdot \nabla S  , \quad 
c_1(\bv;T,S):= \int_\Omega (\bv \cdot \nabla T )S ,  \\ 
c_2(\bu,\bv;S): =   \frac{\mu}{\kappa c'\rho} \int_\Omega (\bu\cdot\bv) \,S, \qquad 
F(S;\bv):= \rho \beta \int_\Omega (S-T_0)\bg \cdot\bv, \qquad 
G(S):= \int_\Omega g\,S, 
  \end{gather*}
  and denote by 
 $\bB_i$ and $\bB_i^*$ the operators induced by the bilinear form $b_i(\cdot,\cdot)$, $i=1,2$: 
\begin{align*}
\bB_1&:\bH_\star(\bcurl_s,\Omega) \to [\bH_\star^r(\vdiv,\Omega)]', \; \langle \bB_1(\bzeta),\bv\rangle := b_1(\bzeta,\bv), \; 
 \bB_1^*:\bH_\star^r(\vdiv,\Omega)\to [\bH_\star(\bcurl_s,\Omega)]',\\
 \bB_2&: \bH^r_\star(\vdiv,\Omega)\to [\rL^2(\Omega)]', \quad  \langle \bB_2(\bv),q\rangle := b_2(\bv,q), \quad 
 \bB_2^*:\rL^2(\Omega) \to [\bH^r_\star(\vdiv,\Omega)]'.
\end{align*} 
The system \eqref{eq:weak} is rewritten as follows: find the tuple $(\bomega,\bu,p,T)\in \bH_\star(\bcurl_s,\Omega) \times \bH^r_\star(\vdiv,\Omega)\times \rL^2(\Omega)  \times \rH_\star^1(\Omega)$ such that 
\begin{subequations}\label{eq:weak2}
\begin{align}
a_1(\bomega,\bzeta) + b_{1}(\bzeta,\bu)&=\;0&\qquad\forall\bzeta\in {\bH_\star(\bcurl_s,\Omega)} \label{eq:weak2-1}  ,\\
 b_1(\bomega,\bv)- a_2(\bu,\bv) + b_2(\bv,p) &=\;F(T;\bv)&\qquad\forall \bv\in{\bH^r_\star(\vdiv,\Omega)}, \\
b_2(\bu,q)&=\;0&\qquad\forall q\in {\rL^2(\Omega)},\label{eq:weak2-3}\\
a_3(T,S)+c_1(\bu;T,S)-c_2(\bu,\bu;S) 
& =\; G (S)&\qquad \forall S\in {\rH_\star^1(\Omega)}.\label{eq:weak2-4}
\end{align}
\end{subequations}

\section{Unique solvability analysis of the continuous formulation}\label{sec:wellp}
The well-posedness analysis of \eqref{eq:weak2} shall use Banach fixed-point theory. To do so we will separate the Darcy--Brinkman equations \eqref{eq:weak2-1}--\eqref{eq:weak2-3} (and will consider a reduced problem in the Kernel of $\bB_2$), from the energy equation \eqref{eq:weak2-4}. 
\subsection{Preliminaries}
Given an arbitrary $t \in (1,\infty)$, consider for each $\bz\in \bL^t(\Omega)$ the function
\begin{equation}\label{def-cal-J-t}
\mathcal J_t(\bz) \,:=\, 
\left\{
\begin{array}{cl}
\displaystyle |\bz|^{t-2}\,\bz & \quad\hbox{if}\,\, \bz \not = \mathbf 0 \,,\\[1ex]
\displaystyle \mathbf 0 & \quad \hbox{otherwise}. 
\end{array}
\right.
\end{equation}
The following two technical results from  \cite{gatica22} provide a useful pairing property, and an appropriate relationship between divergence-free vector fields, respectively.  The second result is here adapted to the case of mixed boundary conditions. 
\begin{lemma}\label{lem-p-q}
Let $t, \, t' \in (1, \infty)$ be such that $\frac1t + \frac{1}{t'} = 1$. Then, for each $\bz\in \bL^t(\Omega)$
there hold 
\begin{subequations}
\begin{gather}\label{eq-lem-p-q-1}
\bw  \,:=\, \mathcal J_t(\bz) \,\in\, \bL^{t'}(\Omega)\,, \qquad \bz \,=\, \mathcal J_{t'}(\bw)\,, \qan \\
\label{eq-lem-p-q-2}
\int_\Omega \bz \cdot \bw = \|\bz\|^t_{\bL^t(\Omega)}
= \|\bw\|^{t'}_{\bL^{t'}(\Omega)} 
= \|\bz\|_{\bL^t(\Omega)} \, \|\bw\|_{\bL^{t'}(\Omega)},
\end{gather}\end{subequations}
and therefore $\mathcal J_t : \bL^t(\Omega) \to \bL^{t'}(\Omega)$
and  $\mathcal J_{t'} : \bL^{t'}(\Omega) \to \bL^t(\Omega)$ are bijective and inverse to each other.
\end{lemma}
\begin{lemma}\label{lem:Ds}
Let $t, \, t' \in (1, \infty)$ be such that $\frac1t + \frac{1}{t'} = 1$. Then, there exists a linear 
and bounded operator $\cD_{t'}:\bL^{t'}(\Omega)\to \bL^{t'}(\Omega)$ satisfying 
\begin{equation}\label{eq-lem-S-1}
\vdiv (\cD_{t'}(\bw)) =0 \quad\text{in}\quad  \Omega \quad \text{and} \quad \cD_{t'}(\bw) \cdot \bn = 0 \quad\text{on}\quad \Gamma, \qquad \forall \, \bw \in \bL^{t'}(\Omega)\,.
\end{equation}
Moreover, for each $\bz \in \bL^t(\Omega)$ such that $\vdiv(\bz) =0$ in $\Omega$ 
and $\bz \cdot \bn = 0 $ on $\Gamma$, there holds 
\begin{equation}\label{eq-lem-S-2}
\int_\Omega \bz \cdot \cD_{t'}(\bw) = \int_\Omega \bz \cdot \bw \qquad \forall \, \bw\in \bL^{t'}(\Omega).
\end{equation}
\end{lemma}

Let us now present, in a form which is convenient for our analysis, the following result  regarding integrability of vector potentials in the case of mixed boundary conditions. 
\begin{lemma}\label{lem:prop}
If $\ff\in \bL^s(\Omega)$ with $\vdiv\ff = 0$ and $(\ff\cdot\bn)|_\Gamma = 0$, then there exists a unique $\bz \in \bW^{1,s}(\Omega)$ solution to 
 \begin{align}
 \bcurl \bz  = \ff \quad \text{and} \quad 
 \vdiv \bz & = 0 \quad \text{in }\Omega,\nonumber\\
 \bz \cdot \bn  &= 0 \quad 
 \text{on  }\Sigma, \label{eq:asterisk}\\
 \bz \times \bn & = \cero \quad \text{on  } \Gamma,\nonumber
 \end{align}
which satisfies  \begin{equation}\label{eq:higher-reg} \|\bcurl \bz \|_{\bL^s(\Omega)} \lesssim \|\ff\|_{\bL^s(\Omega)}.\end{equation}
\end{lemma}
We stress that a proof of a similar result is found in \cite[Theorem 4.21]{amrouche21}, stating  that if the domain is  as in \cite[Section 2]{amrouche21} (but in particular, this requires that $\Gamma$ and $\Sigma$ do not actually meet) and if it is furthermore of class $C^{2,1}$, then the result holds for $1<s< \infty$. Another similar result is given in \cite{jochmann97} but it needs that the right-hand side data is in $\bL^2(\Omega)$. For Lipschitz domains, \cite{jakab09} shows that the required regularity holds when $|2-s|<\epsilon$. 
Restricting the type of domains to open polyhedra, \cite[Section 5]{mazya07} and \cite[Section 4.6]{mazya09} show a closely related, intermediate result for Stokes and Navier--Stokes equations. 
If one takes either Dirichlet or Neumann velocity boundary conditions on each face of the polyhedra, then for $\ff \in \bL^s(\Omega)$ the Stokes velocity belongs to $\bW^{2,s}(\Omega)$ with $1<s\leq \frac{8}{7}$. A slight modification of   \cite[Theorem 5.5]{mazya07} 
 allows us to extend the range and use the Lebesgue exponent $s = \frac{6}{5}$, which will match the  admissible exponent  needed in our subsequent analysis. 

A sketch of the required steps is as follows (we do not provide all details, this is part of \cite{rossmann23}). Note first that every velocity-pressure solution in $\bW^{1,2}(\Theta)\times \rL^2(\Theta)$ is  also in the weighted space $\mathcal{\boldsymbol{W}}_{\beta',0}^{1,2} \times \mathcal{W}_{\beta',0}^{0,2}$ if $\beta'$ is nonnegative and $\Theta$ is bounded. 
Next it is necessary to show (as in \cite[Theorem 4.10]{mazya07}) that the 
lines $\mathrm{Re}\,\lambda= -1/2$ and $\mathrm{Re}\,\lambda=2-0-3/s$ coincide, and this line is free of eigenvalues of the pencils $\{\mathfrak A\}_j(\lambda)$. This can be proven as in \cite[Theorem 3.1]{kozlov98}. Moreover, if the edge angles
$\theta_k$ are less than $\pi$, then we have $\mu_k >1/2$, since 
the equation $\cos(\lambda\theta_k) (\lambda^2\sin^2(\theta_k) -\cos^2(\lambda\theta_k))=0$ has no solutions with real part in the interval $[0,\frac12]$ 
(as was written after \cite[Lemma 2.8]{mazya07}). Then the assumption on the eigenvalues of the pencils $\{\mathfrak A\}_j(\lambda)$ in \cite[Theorem 5.5]{mazya07} readily holds for the non-weighted spaces (with $\beta=\beta'=0$, $\delta = 0$ and $s=\frac65$).  This confirms that the condition ${\max(2-\mu_k,0)}<0+2/s <2$ in  \cite[Theorem 5.5]{mazya07} is also satisfied,
and we can conclude that the velocity-pressure pair belongs to $\mathcal{\boldsymbol{W}}_{0,0}^{2,s} \times \mathcal{W}_{0,0}^{1,s}$. 


\begin{proof} (\textit{Proof of Lemma~\ref{lem:prop}})
Let us assume that $\Omega$ is polyhedral, with all dihedral angles less than $\pi$, and that each face lies in either $\Sigma$ or $\Gamma$. From  the modification of \cite[Section 5]{mazya07} outlined above, 
if $\ff\in \bL^s(\Omega)$ 
with $1<s\leq\frac{6}{5}$, 
then there exists a unique $(\hat{\bz},\pi) \in \bW^{1,s}(\Omega)\times \rL^s(\Omega)$ solution to the Stokes problem with mixed boundary conditions 
 \begin{align}
 -\bDelta \hat{\bz} + \nabla \pi & = \ff \quad \text{in }\Omega,\nonumber\\
 \vdiv \hat{\bz} & = 0 \quad \text{in }\Omega,\nonumber\\
 \hat{\bz}   &= \cero\quad \text{on  }\Sigma, \label{eq:aux-stokes2}\\
 \bnabla\hat{\bz}  \bn & = \cero \quad \text{on  } \Gamma.\nonumber
 \end{align}
Moreover, from the same reference  
we have that  $\hat{\bz} \in \bW^{2,s}(\Omega)$ and 
 \begin{equation}\label{eq:higher-reg-stokes} 
  \|\bcurl \hat{\bz} \|_{\bW^{1,s}(\Omega)} \leq \|\bnabla \hat{\bz} \|_{\mathbb{W}^{1,s}(\Omega)} \lesssim \|\ff\|_{\bL^s(\Omega)}.\end{equation}
Assuming now that $\vdiv\ff = 0$ we readily infer that pressure is zero. Also $\hat{\bz} \times \bn = \cero$ on $\Sigma$. Since on each face lying on $\Gamma$ the shape operator on that surface is zero, then we have 
$(\bnabla\hat{\bz})\bn = \bcurl \hat{\bz} \times \bn = \cero$ on $\Gamma$ (see, e.g., \cite[Section 3.1.2]{kim21}).

Then, thanks to the above considerations, we can define (uniquely, thanks to the uniqueness of $\hat{\bz}$), $\bz \equiv \bcurl \hat{\bz}$ belonging to $\bW^{1,s}(\Omega)$, that satisfies $\bz\times \bn = \cero$ on $\Gamma$, 
$\vdiv\bz = 0$ and $-\bDelta \bz = \bcurl \bcurl \bz+ \bnabla(\vdiv\bz) = \ff$ in $\Omega$. This shows existence and uniqueness of solution to \eqref{eq:asterisk}, and by virtue of the continuous dependence on data \eqref{eq:higher-reg-stokes}, we  have 
\begin{equation}\label{eq:higher-aux}
\|\bz\|_{\bW^{1,s}(\Omega)} \lesssim \|\ff\|_{\bL^s(\Omega)},
\end{equation}
which, from norm definitions, implies \eqref{eq:higher-reg}. 
\end{proof}

We continue by collecting key properties of the bilinear and trilinear forms. 
\subsection{Properties of bilinear and trilinear forms}
First, it is straightforward to see that, thanks to H\"older and Cauchy--Schwarz inequalities, the estimate 
\begin{equation}\label{eq:0-r}
\|\cdot\|_{0,\Omega} \leq C_{\Omega,r} \|\cdot\|_{\bL^r(\Omega)},
\end{equation}    
with $C_{\Omega,r}:=|\Omega|^{(r-2)/2r}$, and the norm definitions \eqref{eq:norms}, the bilinear forms $a_i(\cdot,\cdot),b_i(\cdot,\cdot)$ are all bounded as follows 
\begin{subequations}
\begin{align}
|a_1(\bomega,\bzeta)  | & \leq \|\bomega\|_{0,\Omega}  \|\bzeta\|_{0,\Omega}  \leq \|\bomega\|_{\bcurl_s,\Omega} \|\bzeta\|_{\bcurl_s,\Omega},\label{eq:bound-a1}\\
|b_1(\bzeta,\bv) | & \leq  \sqrt{\mu'} \|\bcurl \bzeta\|_{\bL^s(\Omega)} \|\bv\|_{\bL^r(\Omega)} \leq  \sqrt{\mu'} \|\bzeta\|_{\bcurl_s,\Omega}\|\bv\|_{r,\vdiv,\Omega} ,\label{eq:bound-b1}\\
|a_2(\bu,\bv) |& \lesssim \frac{\mu}{\kappa} \|\bu\|_{\bL^r(\Omega)} \|\bv\|_{\bL^r(\Omega)} \leq \frac{\mu}{\kappa} \|\bu\|_{r,\vdiv,\Omega} \|\bv\|_{r,\vdiv,\Omega},\label{eq:bound-a2}\\
|b_2(\bv,q)| & \leq \|\bv\|_{0,\Omega}\|q\|_{0,\Omega} \lesssim   \|\bv\|_{r,\vdiv,\Omega}\|q\|_{0,\Omega},\label{eq:bound-b2}\\
|a_3(T,S)  | & \leq\max\{\sigma_0,\alpha\} \|T\|_{1,\Omega} \|S\|_{1,\Omega}.\label{eq:bound-a3}
\end{align}\end{subequations}
Similarly the (affine,linear)-form $F:\rH^1(\Omega)\times \bH^r_\star(\vdiv,\Omega)\to \mathbb{R}$ and the linear functional $G:\rH^1(\Omega)\to \mathbb{R}$ are bounded:
\begin{subequations}
\begin{align} 
|F(S;\bv)| &\lesssim \rho \beta \|\bg\|_{\bL^\infty(\Omega)} \|S\|_{1,\Omega}\|\bv\|_{r,\vdiv,\Omega},\label{eq:F-bound}\\
|G(S)| &\leq \|g\|_{0,\Omega}\|S\|_{1,\Omega}.\label{eq:G-bound}
\end{align} \end{subequations}
In particular, note that $F$ may be viewed as an affine map $F:\rH^1_\star(\Omega) \to \bH^r_\star(\vdiv,\Omega)'$ in which case we may express it as the sum $F = \widetilde{F} + F_0$, where $\widetilde{F}: \rH^1_\star(\Omega) \to \bH^r_\star(\vdiv,\Omega)'$ denotes the linear part of $F$.

\noindent Continuing, we recall that $\rH^1(\Omega)$ is continuously embedded into $\rL^t(\Omega)$ with $t \in (1, \infty)$ in $\mathbb{R}^2$ and $t\in (1,6]$ in $\mathbb{R}^3$. More precisely, we have the following inequality
\begin{equation}\label{eq:Sobolev-inequality}
\|w\|_{\rL^t(\Omega)}\leq   C_{\mathrm{Sob}}\,\|w\|_{1,\Omega}\quad 
\forall\,w \in \rH^1(\Omega), 
\end{equation}
with $C_{\mathrm{Sob}}> 0$ depending only on $|\Omega|$ and $t$ (see \cite[Theorem 1.3.4]{quarteroni94}). Then we note that, by virtue of H\"older's inequality, the estimate \eqref{eq:0-r} and \eqref{eq:Sobolev-inequality}, we have that the trilinear forms $c_1(\cdot;\cdot,\cdot)$  and $c_2(\cdot,\cdot;\cdot)$ are bounded as follows 
\begin{subequations}\label{eq:c1-c2-bound}
\begin{align}
|c_1(\bv;T,S)| & \leq \|\bv\|_{\bL^r(\Omega)}\|\nabla T\|_{0,\Omega}\|S\|_{\rL^{t}(\Omega)} \lesssim \|\bv\|_{r,\vdiv,\Omega} \|T\|_{1,\Omega}\|S\|_{1,\Omega},\label{bound-c1}\\
|c_2(\bu,\bv;S) |& \leq \frac{\mu}{\kappa c'\rho} \|\bu\|_{\bL^r(\Omega)} \|\bv\|_{\bL^r(\Omega)} \|S\|_{\rL^{t}(\Omega)} \lesssim \frac{\mu}{\kappa c'\rho} \|\bu\|_{r,\vdiv,\Omega} \|\bv\|_{r,\vdiv,\Omega} \|S\|_{1,\Omega} \label{bound-c2}.
\end{align}\end{subequations}
\begin{remark}
    At this point we can choose a feasible value for the Lebesgue exponents {$r>2$ and $\frac1r+\frac1s=1$}  which is motivated in particular by the boundedness stated in \eqref{eq:bound-b1} and \eqref{bound-c1},  which is valid in both 2D and 3D.  Next, and owing to the   regularity needed for the proof of inf-sup condition for $b_1(\cdot,\cdot)$ (cf.  Lemma~\ref{lem:inf-sup}, below), we specify {$r=6$}.   In this case we take {$s = \frac{6}{5}$} and we can choose the exponent {$t=3$} in \eqref{bound-c1}, and {$t=\frac{3}{2}$} in \eqref{bound-c2}. For the sake of notation, we maintain the indexes $r$ and $s$ in the spaces below, but we stress that we restrict the analysis to the aforementioned specification. 
\end{remark}

Let us denote by  $\bV_0$ the Kernel of the operator $\bB_2$
\begin{align}\label{eq:charact}
    \bV_0& := \mathrm{Ker}(\bB_2) = \{ \bv\in \bH^r(\vdiv,\Omega): b_2(\bv,q) = 0 \quad \forall q\in \rL^2(\Omega)\} \nonumber \\
    & = \{ \bv\in \bH^r(\vdiv,\Omega): \vdiv\bv = 0\}.
\end{align}
We also recall that using the chain rule for the term $\vdiv(T\bv)$, applying integration by parts, and using the boundary conditions for temperature and velocity, the advection term can be written in the anti-symmetric form 
\[ c_1(\bv;T,S) = \frac12 \int_\Omega (\bv \cdot \nabla T )S - \frac12 \int_\Omega (\bv \cdot \nabla S )T \qquad \forall \bv \in \bV_0,\ T,S \in \rH_\star^1(\Omega),\]
and therefore it satisfies the following well-known property 
\begin{equation}\label{eq:c1-zero} c_1(\bv;S,S) = 0 \qquad \forall \bv \in \bV_0,\ S \in \rH_\star^1(\Omega).\end{equation}

We continue with the following Lemma, concerning inf-sup conditions for the bilinear forms $b_i(\cdot,\cdot)$, $i=1,2$.
\begin{lemma}\label{lem:inf-sup}
Assume that the domain {has a polyhedral boundary}. Then there exist positive constants $\beta_1,\beta_2$ such that 
\begin{subequations}
\begin{align}
  \label{eq:inf-sup-b1}
    \sup_{\bzeta \in \bH_\star(\bcurl_s,\Omega)\setminus\{\cero\}} \frac{b_1(\bzeta,\bv)}{\|\bzeta\|_{\bcurl_s,\Omega}} & \geq \beta_1 \|\bv\|_{r,\vdiv,\Omega} \qquad \forall \bv\in \bV_0,\\
    \sup_{\bv \in \bH_\star^r(\vdiv,\Omega)\setminus\{\cero\}} \frac{b_2(\bv,q)}{\|\bv\|_{r,\vdiv,\Omega}} & \geq \beta_2 \|q\|_{0,\Omega} \qquad \forall q\in \rL^2(\Omega).\label{eq:inf-sup-b2}
\end{align}\end{subequations}
\end{lemma}
\begin{proof}
For a given $\bv\in \bV_0$, using Lemmas~\ref{lem-p-q} and \ref{lem:Ds} with $t=r,t'=s$, we can construct 
\[ \cD_s[\mathcal{J}_r(\bv)] := \cD_s(|\bv|^{r-2}\bv)  \in \bL^{s}(\Omega),\]
which furthermore satisfies  $\vdiv (\cD_s[\mathcal{J}_r(\bv)]) = 0$ (due to Lemma~\ref{lem:Ds}) and $(\cD_s[\mathcal{J}_r(\bv)] \cdot \bn)_\Gamma = 0$. Then, owing to Lemma~\ref{lem:prop} (with $\ff = \cD_s[\mathcal{J}_r(\bv)]$ and $t=r$, $t'=s$), we know that there exists $\bz \in \bW^{1,s}(\Omega)$ such that 
\[\bcurl\bz = \cD_s[\mathcal{J}_r(\bv)], \quad \vdiv\bz = 0 \quad \text{in $\Omega$}, \quad 
\bz \times \bn = \cero \quad \text{on $\Gamma$}, \quad \bz\cdot\bn = 0\quad  \text{on $\Sigma$},\]
and it satisfies the estimate 
\begin{equation}\label{eq:bound-7}\|\bcurl\bz\|_{\bL^{s}(\Omega)} 
\lesssim \|\cD_s[\mathcal{J}_r(\bv)]\|_{\bL^{s}(\Omega)}.\end{equation}
By virtue of the Sobolev embedding theorem (and recalling that $r>2$ and that $\frac1r+\frac1s=1$), there holds  $\bz \in \bW^{1,s}(\Omega) \hookrightarrow \bL^r(\Omega) \subset \bL^2(\Omega)$, and we can conclude that $\bz \in \bH_\star(\bcurl_s,\Omega)$.

Next, and thanks again to Lemma~\ref{lem-p-q}, we have that 
\begin{align*}  \sup_{\bzeta \in \bH_\star(\bcurl_s,\Omega)\setminus\{\cero\}} \frac{b_1(\bzeta,\bv)}{\|\bzeta\|_{\bcurl_s,\Omega}} & \geq \frac{b_1(\bz, \bv)}{\|\bz\|_{\bcurl_s,\Omega}} = \frac{\int_\Omega \cD_s[\mathcal{J}_r(\bv)]  \cdot \bv}{\|\bz\|_{\bcurl_s,\Omega}} \\
& = \frac{\| \cD_s[\mathcal{J}_r(\bv)] \|_{\bL^{s}(\Omega)} \|\bv\|_{\bL^r(\Omega)}}{\|\bz\|_{\bcurl_s,\Omega}} \\
& \gtrsim  \|\bv\|_{\bL^r(\Omega)},\end{align*}
where we have used \eqref{eq:bound-7}. The inf-sup constant in \eqref{eq:inf-sup-b1} depends then on the hidden elliptic regularity constant, on the exponents $r=6$, $s=\frac65$ using \eqref{eq:0-r},  and on the dimension $d$.

On the other hand, for the inf-sup condition of $b_2(\cdot,\cdot)$, we use the usual regularity enjoyed by the Stokes equations with mixed boundary conditions: for a given $q\in \rL^2(\Omega)$, find $\br,z $ 
such that 
\[ -\bDelta \br + \nabla z = \cero \quad \text{and} \quad \vdiv \br = q \quad \text{in $\Omega$}, \qquad 
\br = \cero \quad \text{on $\Gamma$}, \qquad (\bnabla \br - z\bI) \bn = \cero \quad \text{on $\Sigma$}.\]
From \cite{girault79} it follows that $\br\in \bH^1_\star(\Omega)$, $z\in \rL^2(\Omega)$ and 
\[\| \br\|_{\bH^1(\Omega)} \lesssim \|q\|_{0,\Omega}.\]
Then we can choose 
\begin{equation}\label{re01}
    \tilde{\bv}\equiv \br \quad \text{with} \quad (\br\cdot\bn)|_\Gamma = 0 \ \text{and} \ \vdiv \tilde{\bv} = q \in \rL^2(\Omega),
\end{equation} 
and using next the Sobolev embedding from $\bH^1(\Omega)$ into $\bL^r(\Omega)$, we can assert that 
\begin{equation}\|\tilde{\bv}\|_{\bL^r(\Omega)}  \lesssim \|\tilde{\bv}\|_{\bH^1(\Omega)} \lesssim \|q\|_{0,\Omega}, \label{re02}\end{equation}
and therefore $\|\tilde{\bv}\|_{r,\vdiv,\Omega}\lesssim \|q\|_{0,\Omega}$. Then \eqref{eq:inf-sup-b2} follows straightforwardly from \eqref{re01}--\eqref{re02} giving the inf-sup constant $\beta_2$ depending on the hidden Stokes  regularity constant, on the dimension,  and on the continuous injection constant.
\end{proof}

Note that  from Lemma~\ref{lem:prop}, for a given $\bv \in \bV_0$, there exists a unique vector potential $\bz$ satisfying $\bcurl\bz = \bv$. Then, and similarly to \eqref{eq:charact}, we can obtain the following characterisation of the null space of $\bB_1^\star$ restricted to the null space of $\bB_2$ 
\begin{equation}\label{eq:char-2}
\text{Ker}(\bB_1^\star|_{\bV_0})=\{\bv\in \bV_0: \int_\Omega \bcurl \bzeta \cdot \bv = 0 \quad \forall \bzeta \in \bH_\star(\bcurl_s,\Omega)\} =\{\cero\}.\end{equation} 

In addition, we shall denote by $\bZ_0$  the Kernel of the operator $\bB_1$
\begin{align}\label{eq:charact-Z0}
 \bZ_0& := \mathrm{Ker}(\bB_1) = \{ \bzeta\in \bH_\star(\bcurl_s,\Omega): b_1(\bzeta,\bv) = 0 \quad \forall \bv\in \bH_\star^r(\vdiv,\Omega)\} \nonumber \\
 & = \{ \bzeta\in \bH_\star(\bcurl_s,\Omega): \bcurl\bzeta = \cero\}.
\end{align}

Finally, and directly from their definition, we can state the positivity and coercivity of the diagonal bilinear forms $a_i(\cdot,\cdot)$, $i =1,2,3$:
\begin{subequations}
    \begin{align}
  \label{eq:a1-coer}  a_1(\bzeta,\bzeta) & =\|\bzeta\|^2_{0,\Omega} =  \|\bzeta\|^2_{\bcurl_s,\Omega}  \qquad \forall\bzeta\in \bZ_0,\\
   \label{eq:a2-pos}     a_2(\bv,\bv) & = {\frac{\mu}{\kappa}\|\bv\|^2_{0,\Omega} \geq 0 \qquad \forall \bv \in \bV_0},  \\
\label{eq:a3-coer}    a_3(S,S) & \geq \min\{\sigma_0,\alpha\} \|S\|^2_{1,\Omega} \qquad \forall S \in \rH_\star^1(\Omega).
    \end{align}
\end{subequations}

We finalise this preliminary section stating an abstract result (unique solvability of perturbed saddle-point problems in Banach spaces) required in the proof of well-posedness of the decoupled Darcy--Brinkman equations. Its proof can be found in the recent work \cite[Theorem 3.4]{correa22}.

\begin{theorem}\label{th:abstract}
Let $X,Y$ be reflexive Banach spaces and consider bounded bilinear forms $a:X\times X\to \mathbb{R}$, $b:X\times Y\to \mathbb{R}$, and $c:Y\times Y\to \mathbb{R}$ (with boundedness constants $\|a\|,\|b\|$, and $\|c\|$, respectively). Let $\bB:X\to Y'$ be a bounded linear map induced by $b(\cdot,\cdot)$ and let $X_0$ denote its null space. Assume that $\mathrm{Ker}(\bB^*)=\{0\}$ and further suppose that 
\begin{enumerate}
    \item $a(\cdot,\cdot)$ and $c(\cdot,\cdot)$ are symmetric and semi-positive definite over $X$ and $Y$, respectively,
    \item there exists $\tilde{\alpha}>0$ such that 
    \[\sup_{\tau\in X_0\setminus\{0\}} \frac{a(\sigma_0,\tau)}{\|\tau\|_X}\geq \tilde{\alpha} \|\sigma_0\|_X \qquad \forall \sigma_0\in X_0,\]
    \item and there exists $\tilde{\beta}>0$ such that 
      \[\sup_{\tau\in X\setminus\{0\}} \frac{b(\tau,v)}{\|\tau\|_X}\geq \tilde{\beta} \|v\|_Y \qquad \forall v\in Y.\]
\end{enumerate}
Then, for each $(f,g)\in X'\times Y'$ there exists a unique $(\sigma,u)\in X\times Y$ solution to the following perturbed saddle-point problem 
\begin{align*}
    a(\sigma,\tau) + b(\tau,u)&=\;f(\tau)\qquad\forall\tau\in X, \\
 b(\sigma,v)- c(u,v) &=\;g(v)\qquad\forall v\in Y.
\end{align*}
Moreover, the solution satisfies the stability bound  
\[\|(\sigma,u)\|_{X\times Y} \lesssim \|f\|_{X'}+\|g\|_{Y'},\]
where the hidden constant depends only on $\|a\|,\|c\|,\tilde{\alpha}$, and $\tilde{\beta}$.
\end{theorem}

\subsection{Solvability of the decoupled Darcy--Brinkman equations}
We start by establishing the wellposedness of the Darcy--Brinkman problem for a given temperature. The analysis for the Hilbertian case and in 2D is performed in \cite{anaya16} using Banach--Ne\v{c}as--Babu\v{s}ka's theory and working in the Kernel of the divergence operator. Here the proof follows instead a perturbed saddle-point argument adapted to the Banach spaces' context and using Theorem~\ref{th:abstract}. 

Let us denote the product space 
\begin{equation}\label{def:frak}
\mathfrak{U} := \bH_\star(\bcurl_s, \Omega) \times \bH_\star^r(\vdiv, \Omega) \times \rL^2(\Omega),\end{equation}
and define the map $\cA:\mathfrak{U} \to \mathfrak{U}'$,  $(\bomega,\bu,p)\mapsto \cA(\bomega,\bu,p)$, implicitly through the following weak Brinkman equation (for a fixed $\widetilde{T} \in \rH^1_\star(\Omega)$)
\begin{align}
    a_1(\bomega,\bzeta) + b_{1}(\bzeta,\bu)&=\;0&\qquad\forall\bzeta\in {\bH_\star(\bcurl_s,\Omega)}   ,\nonumber \\
 b_1(\bomega,\bv)- a_2(\bu,\bv) + b_2(\bv,p) &=\;F(\widetilde{T};\bv)&\qquad\forall \bv\in{\bH^r_\star(\vdiv,\Omega)}, \label{eq:decoupled-Brinkman} \\
b_2(\bu,q)&=\;0&\qquad\forall q\in {\rL^2(\Omega)}.\nonumber
\end{align}
\begin{lemma}\label{lem:Brinkman}
For a fixed $\widetilde{T} \in \rH^1_\star(\Omega)$, there exists a unique tuple $(\bomega,\bu,p) \in \mathfrak{U}
$ such that the operator equation $\mathcal{A}(\bomega,\bu,p) = (0,F(\widetilde{T}),0)$ defined by \eqref{eq:decoupled-Brinkman} is satisfied. Moreover, there holds 
\begin{equation}\label{eq:Brinkman-cont-dep}
 \|\bomega\|_{\bcurl_s,\Omega} + \|\bu\|_{r,\vdiv,\Omega} + \|p\|_{0,\Omega} \lesssim \bigl(2 + \sqrt{\mu'}{+\frac{\mu}{\kappa}}\bigr) \rho \beta \|\bg\|_{\bL^\infty(\Omega)} \|\widetilde{T}\|_{1,\Omega}, 
\end{equation}
where the hidden constant depends on the inf-sup constants $\beta_1,\beta_2$.  
\end{lemma}
Before addressing the unique solvability of \eqref{eq:decoupled-Brinkman}, note that for any $\widetilde{T} \in \rH^1_\star(\Omega)$ we have $F(\widetilde{T}) \in (\bH^r_\star(\vdiv,\Omega))'$ granted by \eqref{eq:F-bound} with 
\[\|F(\widetilde{T})\|_{(\bH^r_\star(\vdiv,\Omega))'} \leq \rho \beta \|\bg\|_{\bL^\infty(\Omega)} \|\widetilde{T}\|_{1,\Omega}.\]
Now, let us consider the following reduced problem where velocity is sought in the Kernel of $\bB_2$: for fixed $\widetilde{T} \in \rH^1_\star(\Omega)$, find $(\bomega,\bu)\in \bH_\star(\bcurl_s,\Omega)\times \bV_0$ such that 
\begin{equation}
\label{eq:decoupled-Brinkman-reduced}
\begin{aligned}    a_1(\bomega,\bzeta) + b_{1}(\bzeta,\bu)&=\;0&\qquad\forall\bzeta\in {\bH_\star(\bcurl_s,\Omega)}   , \\
 b_1(\bomega,\bv)- a_2(\bu,\bv) &=\;F(\widetilde{T};\bv)&\qquad\forall \bv\in \bV_0.
\end{aligned}\end{equation}

We stress that if $(\bomega,\bu,p)\in \bH_\star(\bcurl_s,\Omega)\times \bH^r_\star(\vdiv,\Omega)\times \rL^2(\Omega)$ is the unique solution to \eqref{eq:decoupled-Brinkman}, then it is evident that the pressure can be eliminated and then $(\bomega,\bu)\in \bH_\star(\bcurl_s,\Omega)\times \bV_0$ is a solution to \eqref{eq:decoupled-Brinkman-reduced}. Conversely, if  $(\bomega,\bu)\in \bH_\star(\bcurl_s,\Omega)\times \bV_0$ solves  \eqref{eq:decoupled-Brinkman-reduced}, then there exists {a unique}  $p\in \rL^2(\Omega)$ such that $(\bomega,\bu,p)$ is a solution to \eqref{eq:decoupled-Brinkman}. {To see this, note from the closed range theorem that the image of $\bB_2^*$ is equivalent to the annihilator of $\bV_0$. Next let us set $\mathcal{G}:\bH_\star^r(\vdiv, \Omega) \to \mathbb{R}$ by the following: 
\begin{equation}\label{eq:G-problem-equiv}
\mathcal{G}\bv := F(\widetilde{T};\bv) - b_1(\bomega,\bv) + a_2(\bu,\bv) \qquad \forall \bv \in \bH^r_\star(\vdiv,\Omega).
\end{equation}
As $\mathcal{G}|_{V_0} \equiv 0$ (cf., the second equation in \eqref{eq:decoupled-Brinkman-reduced}), there exists $p \in \rL^2(\Omega)$ yielding $\bB_2^*p = \mathcal{G}$. The uniqueness of $p$ follows from the inf-sup condition \eqref{eq:inf-sup-b2}. Then given $\bu\in \bV_0$, it is clear that 
\[b_2(\bu,q) = 0 \qquad \forall q\in \rL^2(\Omega),\]
and thus the equivalence between \eqref{eq:decoupled-Brinkman} and \eqref{eq:decoupled-Brinkman-reduced} is shown.}


\begin{proof} (\textit{Proof of Lemma~\ref{lem:Brinkman}})
First we note that if we restrict ourselves to the reduced problem (in the Kernel of $\bB_2$), then the Kernel of $\bB_1^*$ is the zero vector (cf. \eqref{eq:char-2}). Next, bearing in mind 
the boundedness of the bilinear forms in \eqref{eq:bound-a1}--\eqref{eq:bound-b2}, the symmetry of the bilinear forms $a_1(\cdot,\cdot)$ and $a_2(\cdot,\cdot)$, the coercivity of $a_1(\cdot,\cdot)$ on the Kernel of the operator $\bB_1$ intersected with that of $\bB_2$ \eqref{eq:a1-coer}, the positivity of $a_2(\cdot,\cdot)$ on the Kernel of the operator $\bB_2$ \eqref{eq:a2-pos}, and the inf-sup conditions from Lemma~\ref{lem:inf-sup}; it follows that conditions (1)--(3) of Theorem~\ref{th:abstract} are met. Therefore there exists a unique  $(\bomega,\bu)\in \bH_\star(\bcurl_s,\Omega)\times \bV_0$ solution to the reduced problem \eqref{eq:decoupled-Brinkman-reduced}, and the continuous dependence on data together with the boundedness of $\widetilde{F}$ imply that 
\begin{equation}
    \label{eq:wu}
\|(\bomega,\bu)\| = \|\bomega\|_{\bcurl_s,\Omega}+ \|\bu\|_{r,\vdiv,\Omega} \leq C_{s_1}\, \rho \beta \|\bg\|_{\bL^\infty(\Omega)} \|\widetilde{T}\|_{1,\Omega},
\end{equation}
with $C_{s_1}$ depending on $C_{\Omega,r}$, $C_{Sob}$, $\beta_1$, $\kappa$ and $\mu$. 
Finally, in order to verify \eqref{eq:Brinkman-cont-dep}, we use again the inf-sup condition \eqref{eq:inf-sup-b2},  as well as the second equation in \eqref{eq:decoupled-Brinkman} to obtain 
\begin{align*} \|p\|_{0,\Omega} & \leq \frac{1}{\beta_2}\sup_{\bv\in \bH^r_\star(\vdiv,\Omega)\setminus\{\cero\}} \frac{b_2(\bv,p)}{\|\bv\|_{r,\vdiv,\Omega}} \\
& \leq  \frac{1}{\beta_2}\sup_{\bv\in \bH^r_\star(\vdiv,\Omega)\setminus\{\cero\}} \frac{|F(\widetilde{T};\bv) - b_1(\bomega,\bv) + a_2(\bu,\bv)|}{\|\bv\|_{r,\vdiv,\Omega}} \\
& {\lesssim} \frac{1}{\beta_2}\bigl(1 + \sqrt{\mu'} {+ \frac{\mu}{\kappa}}\bigr) \rho \beta \|\bg\|_{\bL^\infty(\Omega)} \|\widetilde{T}\|_{1,\Omega},
\end{align*}
where for the last step we have used {triangle inequality and the boundedness properties \eqref{eq:bound-b1} and \eqref{eq:bound-a2}}, together with \eqref{eq:wu}. 
\end{proof}

\subsection{Solvability of the decoupled thermal energy equations}
The well-posedness of the temperature equation (for a given vorticity and velocity) is addressed next. It is a straightforward consequence of the decoupling assumptions and of the Lax--Milgram lemma. 
\begin{lemma}\label{lem:energy}
 For a fixed $\tilde{\bu}\in\bV_0$,
 there exists a unique $T\in \rH^1_\star(\Omega)$ such that 
\begin{equation}\label{eq:decoupled-energy}
    a_3(T,S)+c_1(\tilde{\bu};T,S) 
    =\; G (S) + c_2(\tilde{\bu},\tilde{\bu};S)\qquad \forall S\in {\rH_\star^1(\Omega)}.
\end{equation}
Furthermore, its solution satisfies the following continuous dependence on data 
\begin{equation}\label{eq:energy-cont-dep}
{\|T\|_{1,\Omega} \leq C_{s_2} \max\{\sigma_0^{-1},\alpha^{-1}\}}
\biggl[\|g\|_{0,\Omega} + \frac{\mu}{c'\rho\kappa} \|\tilde{\bu}\|^2_{r,\vdiv,\Omega}
\biggr],
\end{equation}
where the constant $C_{s_2}$ depends on $C_{Sob}$, associated with the injection that leads to \eqref{bound-c1}--\eqref{bound-c2}.
\end{lemma}
\begin{proof}
Similarly as in the proof of Lemma~\ref{lem:Brinkman} above, since the pair $(\tilde{\bomega},\tilde{\bu})\in \bH_\star(\bcurl_s,\Omega)\times \bV_0$ is given, the trilinear form $c_1(\cdot;\cdot,\cdot)$ can be regarded as a bilinear form and the trilinear form $c_2(\cdot,\cdot;\cdot)$  turns out to be a linear functional in $(\rH^1_\star(\Omega))'$. Then, the boundedness of $a_3(\cdot,\cdot)$ and of $c_1(\cdot,\cdot;\cdot)$, the coercivity \eqref{eq:a3-coer} and the skew-symmetry of $c_1(\cdot,\cdot;\cdot)$ for $\tilde{\bu}\in \bV_0$ stated in \eqref{eq:c1-zero} imply, thanks to the Lax--Milgram lemma, that there exists a unique $T$ solution to \eqref{eq:decoupled-energy}. On the other hand, the verification of the bound \eqref{eq:energy-cont-dep} readily follows from the boundedness of $c_2(\cdot,\cdot;\cdot)$, from the {coercivity constant} of $a_3(\cdot,\cdot)$, and from the estimate \eqref{eq:G-bound}. 
\end{proof}
\subsection{A fixed-point approach}
We will appeal to the Banach fixed-point theorem. For this we follow the steps used in, e.g., \cite{gatica22}. Let us now define the following solution operator 
\begin{align*}\cS_1: \ \rH^1_\star(\Omega) & \to \bH_\star(\bcurl_s,\Omega)\times \bV_0 , \\
 \widetilde{T} & \mapsto \cS_1(\widetilde{T})= \bigl(\cS_{11}(\widetilde{T}),\cS_{12}(\widetilde{T})\bigr) := (\bomega,\bu),
\end{align*}
where $(\bomega,\bu)$ is the unique solution to \eqref{eq:decoupled-Brinkman-reduced}, confirmed thanks to  Lemma~\ref{lem:Brinkman} and the equivalence between problems \eqref{eq:decoupled-Brinkman} and \eqref{eq:decoupled-Brinkman-reduced}; and the solution operator 
\[\cS_2: 
\bV_0 \to \rH^1_\star(\Omega), \qquad 
\tilde{\bu}\mapsto \cS_2(\tilde{\bu})
:= T, 
\]
where $T$ is the unique solution to \eqref{eq:decoupled-energy}, according to Lemma~\ref{lem:energy}. Owing to the well-definition of these solution operators we can properly define the operator 
\[\cF: \rH^1_\star(\Omega)\to \rH^1_\star(\Omega), \qquad 
T \mapsto \cF(T) := [\cS_2\circ\cS_{12}](T),\]
and observe that the nonlinear problem \eqref{eq:weak2} is thus equivalent to the following fixed-point equation:
\begin{equation}\label{eq:fixed-point}
\text{Find $T\in \rH^1_\star(\Omega)$ such that $\cF(T) = T$}.\end{equation}

Let us define the following data-dependent constants
\begin{equation}\label{def:C1}
C_1:= C_{s_2}
{\max\{\sigma_0^{-1},\alpha^{-1}\}
\|g\|_{0,\Omega}} \qan  
C_2:=C_{s_1}^2 C_{s_2}{\max\{\sigma_0^{-1},\alpha^{-1}\}
(\rho \beta \|\bg\|_{\bL^\infty(\Omega)})^2}
\frac{\mu}{c'\rho\kappa} 
.
\end{equation}
\begin{lemma}\label{lem:F-map-b}
Assume that
\begin{equation}\label{eq:def-r}
C_1C_2 < \frac14,
\end{equation}
and denote by $x_1:=\frac{1-\sqrt{1-4C_1 C_2}}{2C_2}$ and $x_2:=\frac{1+\sqrt{1-4C_1 C_2}}{2C_2}$ the solutions of the equation $C_2 x^2 -x +C_1 =0$. Then, given {$R>0$} such that $x_1\leq R \leq x_2$, we have that $\cF$ maps the following closed ball in $\rH^1_\star(\Omega)$ into itself 
\[ \rY^R :=\{ S \in \rH^1_\star(\Omega): \|S\|_{1,\Omega} \leq R\}. \]
\end{lemma}
\begin{proof}
Recalling  the continuous dependence on data from \eqref{eq:wu} and \eqref{eq:energy-cont-dep}, and taking $\widetilde{T} \in\rY^R$, and $(\wt\bomega,\wt\bu)=(\cS_{11}(\widetilde{T}),\cS_{12}(\widetilde{T}))$, we have that  
\[\|(\cS_{11}(\widetilde{T}),\cS_{12}(\widetilde{T}))\| \leq C_{s_1}\, \rho \beta \|\bg\|_{\bL^\infty(\Omega)} \|\widetilde{T}\|_{1,\Omega} \leq C_{s_1}\,\rho \beta \|\bg\|_{\bL^\infty(\Omega)} R,\]
and 
\begin{align*} \| \widetilde{T} \|_{1,\Omega} & 
\leq C_{s_2} {\max\{\sigma_0^{-1},\alpha^{-1}\}}\bigl[\|g\|_{0,\Omega} + \frac{\mu}{c'\rho\kappa} \|\tilde{\bu}\|^2_{r,\vdiv,\Omega}
\bigr], \\
& 
\leq C_{s_2}
{\max\{\sigma_0^{-1},\alpha^{-1}\}\biggl[\|g\|_{0,\Omega} + \frac{\mu}{c'\rho\kappa} 
\|(\cS_{11}(\widetilde{T}),\cS_{12}(\widetilde{T}))\|^2\biggr]},
\end{align*}
respectively. Then, appealing to 
the definition of $\cF(\widetilde{T}) := [\cS_2\circ
\cS_{12}
](\widetilde{T})$ and the definition of the constants in \eqref{def:C1}, we can assert that 
\begin{align*}
\|\cF(\widetilde{T}) \|_{1,\Omega} & =    \|\cS_2(
\cS_{12}(\widetilde{T})) \|_{1,\Omega} \\
& \leq C_{s_2} 
{\max\{\sigma_0^{-1},\alpha^{-1}\}
\biggl[\|g\|_{0,\Omega} + 
\frac{\mu}{c'\rho\kappa} 
(C_{s_1}\,\rho \beta \|\bg\|_{\bL^\infty(\Omega)} R)^2\biggr]}\\ 
& 
= {C_1} 
+ C_2 R^2\\
& \leq R,  
\end{align*}
where the last line comes from $x_1 \leq R \leq x_2$ and 
the assumption \eqref{eq:def-r}. This completes the proof.
\end{proof}

\begin{lemma}
 The map $\cF$ is Lipschitz continuous in a neighbourhood of the origin.   
\end{lemma}
\begin{proof}
For $T_1,T_2\in \rH^1_\star(\Omega)$ consider the unique solutions $(\bomega_1,\bu_1)$ and $(\bomega_2,\bu_2)$ associated with each problem of the type \eqref{eq:decoupled-Brinkman-reduced}. Subtracting the resulting problems and using the bilinearity of $\widetilde{F}$, we have that $(\bomega_1-\bomega_2,\bu_1-\bu_2)$ is the unique solution of the reduced problem 
\begin{align*}
    a_1(\bomega_1-\bomega_2,\bzeta) + b_{1}(\bzeta,\bu_1-\bu_2)&=\;0&\qquad\forall\bzeta\in {\bH_\star(\bcurl_s,\Omega)}   ,\nonumber \\
 b_1(\bomega_1-\bomega_2,\bv)- a_2(\bu_1-\bu_2,\bv)  &=\;\widetilde{F}(T_1-T_2;\bv)&\qquad\forall \bv\in{\bV_0}.
\end{align*}
Then, from the definition of the solution operator $\cS_1$ we can readily obtain  
\begin{align}
\nonumber    \|\cS_1(T_1)-\cS_{1}(T_2)\| & =  \|(\bomega_1,\bu_1)-(\bomega_2,\bu_2)\| \\
\nonumber    & \leq \|\bomega_1-\bomega_2\|_{\bcurl_s,\Omega}+ \|\bu_1-\bu_2\|_{r,\vdiv,\Omega} \\
    & \leq 
    C_{s_1}\rho\beta\|\bg\|_{\bL^\infty(\Omega)}\|T_1-T_2\|_{1,\Omega},\label{eq:S1-lip}
\end{align}
where we have used the continuous dependence on data \eqref{eq:wu}. 

Analogously as above, for $(\bomega_1,\bu_1),(\bomega_2,\bu_2) \in \bH_\star(\bcurl_s,\Omega)\times \bV_0$ 
let $T_1,T_2 \in \rH^1_\star(\Omega)$ be the unique solutions to each problem of the type \eqref{eq:decoupled-energy}. Subtracting these problems gives us 
\begin{align*}
     a_3(T_1-T_2,S)+c_1(\bu_1;T_1,S)-c_1(\bu_2;T_2,S)-c_2(\bu_1,\bu_1;S)    +c_2(\bu_2,\bu_2;S)     &=0 \quad \forall S\in {\rH_\star^1(\Omega)},
\end{align*}
and after adding and subtracting the terms $c_1(\bu_1;T_2,S)$ and  $c_2(\bu_1,\bu_2;S)$, then taking $S = T_1 - T_2 \in \rH_\star^1(\Omega)$ as test function, and using the coercivity of the bilinear form $a_3(\cdot,\cdot)$ together with the property \eqref{eq:c1-zero}, we get 
\begin{equation*}
\begin{array}{l}
\min\{\sigma_0,\alpha\} \|T_1-T_2\|^2_{1,\Omega} \leq a_3(T_1-T_2,T_1-T_2) \\
\ds \qquad = -c_1(\bu_1-\bu_2;T_2,T_1-T_2)+c_2(\bu_1-\bu_2;\bu_2,T_1-T_2)+c_2(\bu_1,\bu_1-\bu_2;T_1-T_2).
\end{array}  
\end{equation*}
Therefore we can combine this estimate with the definition of the map $\cS_2$ and the boundedness properties \eqref{bound-c1}--\eqref{bound-c2}, and then divide by $\|T_1-T_2\|_{1,\Omega}$ on both sides of the inequality to arrive at 
\begin{align}
\nonumber  &   \|\cS_2(\bu_1)-\cS_2(\bu_2)\|_{1,\Omega}  =  \|T_1-T_2\|_{1,\Omega} \\
    &\quad \leq C_{s_2}\max\{\sigma_0^{-1},\alpha^{-1}\}( \|T_2\|_{1,\Omega}+ \|\bu_1\|_{r,\vdiv,\Omega} + \|\bu_2\|_{r,\vdiv,\Omega} )\|\bu_1-\bu_2\|_{r,\vdiv,\Omega}.\label{eq:S2-lip}
\end{align}
Now, let $T_1, T_2, \wt T_1, \wt T_2\in \rY^R$, be such that $\wt T_1 = \cF(T_1)$ and $\wt T_2 = \cF(T_2)$.
According to the definition of $\cF$, from \eqref{eq:S1-lip} and \eqref{eq:S2-lip}, using that $\|(\cS_{11}(T_1),\cS_{12}(T_1))\|$ and $ \|(\cS_{11}(T_2),\cS_{12}(T_2))\|$ satisfy \eqref{eq:wu}, and $T_1,\, T_2, \,\wt T_2\in \rY^R$,  we deduce that
\begin{align}\label{eq:F-lip}
&\|\cF(T_1) - \cF(T_2) \|_{1,\Omega}  =    \|\cS_2(\cS_{12}(T_1)) - \cS_2(\cS_{12}(T_2))\|_{1,\Omega} \nonumber \\
&\leq C_{s_2}\max\{\sigma_0^{-1},\alpha^{-1}\}\Big(\|\wt T_2\|_{1,\Omega}+ \|\cS_{12}(T_1)\|_{r,\vdiv,\Omega} + \|\cS_{12}(T_2)\|_{r,\vdiv,\Omega} \Big)\|\cS_{12}(T_1)-\cS_{12}(T_2)\|_{r,\vdiv,\Omega} \nonumber\\
&\leq C_{s_2}\max\{\sigma_0^{-1},\alpha^{-1}\}\Big( \|\wt T_2\|_{1,\Omega}+ C_{s_1}\rho \beta \|\bg\|_{\bL^\infty(\Omega)} (\|T_1\|_{1,\Omega} + \|T_2\|_{1,\Omega} )\Big)\|\cS_{12}(T_1)-\cS_{12}(T_2)\|_{r,\vdiv,\Omega} \nonumber\\
&\leq R\,C_{s_2} \max\{\sigma_0^{-1},\alpha^{-1}\}( 1 + 2C_{s_1} \rho \beta \|\bg\|_{\bL^\infty(\Omega)}) C_{s_1}\rho\beta\|\bg\|_{\bL^\infty(\Omega)} \|T_1-T_2\|_{1,\Omega}.
\end{align}
\end{proof}

We are ready now to prove the main result of this section, that is, the existence and uniqueness of solution of problem \eqref{eq:weak2-1}--\eqref{eq:weak2-4}. 

\begin{theorem}
 Assume that $C_1C_2 < 1/4$, where $C_1,C_2$ are as in \eqref{def:C1}. Then, given $R>0$ such that
\begin{equation}\label{eq:small-r}
 x_1\leq R\leq x_2 \qan R\,C_{s_1} C_{s_2}\rho\beta\|\bg\|_{\bL^\infty(\Omega)} \max\{\sigma_0^{-1},\alpha^{-1}\}( 1 + 2C_{s_1} \rho \beta \|\bg\|_{\bL^\infty(\Omega)}) <1,
\end{equation} 
there exists a unique solution to \eqref{eq:fixed-point}, and equivalently, to \eqref{eq:weak2}.    
\end{theorem}
\begin{proof}
    We recall from the previous analysis that the first assumption in \eqref{eq:small-r} ensures that $\cF$ maps $Y^R$ into itself (see Lemma~\ref{lem:F-map-b}). In addition, from \eqref{eq:F-lip} and the second assumption in \eqref{eq:small-r}, we have that $\cF$ is a contraction mapping, which together with the Banach fixed-point theorem implies that $\cF$ has a unique fixed point in $Y^R$. The proof then follows from the definition of the fixed-point map. 
\end{proof}

\section{Galerkin scheme and well-posedness of the discrete problem}\label{sec:FE}
Here we derive a discrete formulation and show, under appropriate assumptions of the bilinear forms and finite dimensional spaces, that the discrete problem has a unique solution. 
\subsection{Preliminaries}
Let $\cT_h$ denote a family of non-degenerate triangular / tetrahedral meshes on $\Omega$ and denote by $\cE_h$ the set of all facets (edges in 2D) in the mesh. 
By $h_K$ we denote the diameter of the element $K$ and by $h_F$ we denote the length/area of the facet $F$. As usual, by $h$ we denote the maximum of the diameters of elements in $\cT_h$. For all meshes we assume that they are sufficiently regular (there exists a uniform positive constant $\eta_1$ such that each element $K$ is star-shaped with respect to a ball of radius greater than $\eta_1 h_K$. It is also assumed that there exists $\eta_2>0$ such that for each element and every facet $F\in \partial K$, we have that $h_F\geq \eta_2 h_K$, see, e.g., \cite{quarteroni09,ern21}). 

Let us consider the following generic finite dimensional subspaces of the trial-test spaces 
\[\bZ_h \subset \bH_\star(\bcurl_s,\Omega), \quad 
\bV_h \subset \bH^r_\star(\vdiv,\Omega), \quad \rQ_h \subset \rL^2(\Omega), \quad \rY_h\subset \rH^1_\star(\Omega),\]
and denote by $\bV_{0,h}$ the discrete Kernel of $b_2(\cdot,\cdot)$, that is 
\[ \bV_{0,h}: = \{ \bv_h\in \bV_h: b_2(\bv_h,q_h) = 0 \quad \forall q_h\in \rQ_h \}.\]
We also take into account the subspace $\bV_{h,s}\subset \bH^s_\star(\vdiv,\Omega)$.

We consider along this section the following assumptions on these spaces  
\begin{enumerate}
[label=(\textbf{A\arabic*}), ref=\textrm{(A\arabic*)}]
   \item \label{A1} there exists $\hat{\beta}_1>0$ independently of $h$ 
such that 
      \[ \sup_{\bzeta_h \in \bZ_h\setminus\{\cero\}} \frac{b_1(\bzeta_h,\bv_h)}{\|\bzeta_h\|_{\bcurl_s,\Omega}}  \geq \hat{\beta}_1 \|\bv_h\|_{r,\vdiv,\Omega} \qquad \forall \bv_h\in \bV_{0,h},\]
  \item \label{A2} there exists $\hat{\beta}_2>0$ independently of $h$ 
such that   \[    \sup_{\bv_h \in \bV_h\setminus\{\cero\}} \frac{b_2(\bv_h,q_h)}{\|\bv_h\|_{r,\vdiv,\Omega}}  \geq \hat{\beta}_2 \|q_h\|_{0,\Omega} \qquad \forall q_h\in \rQ_h,\]
    \item \label{A3} $\vdiv\bV_h \subseteq \rQ_h$. 
    \item \label{A4} {$\bcurl \bZ_h \subseteq \bV_{h,s}$}.
\end{enumerate}
The following characterisation is a consequence of the definition of the bilinear form $b_2(\cdot,\cdot)$ and  assumption \ref{A3}:  
\[ \bV_{0,h}= \{ \bv_h\in \bV_h: \vdiv\bv_h = 0\}.\]
Similarly, from assumption \ref{A4} we have that 
\[ \bZ_{0,h} =  \{ \bzeta\in \bZ_h: \bcurl\bzeta_h = \cero\},\]
which implies that the discrete Kernel is contained in the continuous one, yielding in turn that 
\[ a_1(\bzeta_h,\bzeta_h) = \|\bzeta_h\|^2_{\bcurl_s,\Omega} \qquad \forall \bzeta_h \in \bZ_{0,h}. \]
The discrete problem is a system of  nonlinear algebraic equations that reads as follows: find $(\bomega_h,\bu_h,p_h,T_h)\in \bZ_h\times\bV_h\times \rQ_h\times \rY_h$ such that
\begin{subequations}\label{eq:galerkin}
\begin{align}
a_1(\bomega_h,\bzeta_h)+b_{1}(\bzeta_h,\bu_h)&=\;0& \forall\bzeta_h\in\bZ_h,\\
b_1(\bomega_h,\bv_h)-a_2(\bu_h,\bv_h)+b_2(\bv_h,p_h) &=\;F(T_h;\bv_h)&\forall\bv_h\in\bV_h,\\
b_2(\bu_h,q_h)&=\;0&\forall q_h\in \rQ_h,\\
a_3(T_h,S_h) + c_1(\bu_h; T_h,S_h) - c_2(\bu_h,\bu_h;S_h) 
& =\; G(S_h)&\forall S_h\in \rY_h.
\end{align}
\end{subequations}

\subsection{Unique solvability of the Galerkin method}
We shall proceed in much the same way as in the continuous case where we decouple the system and seek a fixed point. First we consider the reduced problem
in the discrete kernel of $\bB_2$. That is, for a fixed $\widetilde{T}_h \in \rY_h$ we seek $(\bomega_h, \bu_h) \in \bZ_h \times \bV_{0,h}$ such that 
\begin{equation}\label{eq:Brinkman-reduced-disc}
\begin{aligned}
a_1(\bomega_h,\bzeta_h) + b_{1}(\bzeta_h,\bu_h)&=\;0&\qquad\forall\bzeta_h\in \bZ_h   , \\
 b_1(\bomega_h,\bv_h)- a_2(\bu_h,\bv_h) &=\;F(\widetilde{T}_h;\bv_h)&\qquad\forall \bv_h\in \bV_{0,h}.
\end{aligned}
\end{equation}
The discrete solution theory for the above is another consequence of results shown in, and particularly by, \cite[Theorem 3.5]{correa22}, which we state here for convenience.
\begin{theorem}\label{th:discrete-abstract}
Let $\{X_h\}_{h > 0}$ and $\{Y_h\}_{h>0}$ be families of finite dimensional subspaces of $X,Y$,
respectively. Consider the discrete restrictions $a:X_h \times X_h \to \mathbb{R}$,
$b:X_h \times Y_h \to \mathbb{R}$, and $c:Y_h \times Y_h \to \mathbb{R}$.
Let $\bB: X_h \to Y_h$ be the discrete induced operator by $b$ and denote its kernel by $X_{0,h}$.
Assume that the discrete transpose $\bB^*$ has trivial kernel such that $\ker(\bB^*) = \{0\}$. 
Then, if
\begin{enumerate}
    \item $a(\cdot,\cdot)$ and $c(\cdot,\cdot)$ are symmetric and semi-positive definite over $X_h$ and $Y_h$, respectively,
    \item there exists constant $\tilde{\alpha}_d > 0$ which is independent of $h$, such that,
    \[
    \sup_{\tau_h \in X_{0,h} \setminus \{0\}} \frac{a(\sigma_{0,h}, \tau_h)}{\norm{\tau_h}_{X_h}}
    \geq \tilde{\alpha}_d \norm{\sigma_0}_{X_h}
    \qquad \forall \sigma_{0,h} \in X_{0,h},
    \]
    \item there exists constant $\tilde{\beta}_d > 0$ which is independent of $h$, such that,
    \[
    \sup_{\tau_h \in X_{0,h} \setminus \{0\}} \frac{b(\tau_h, v_h)}{\norm{\tau_h}_{X_h}}
    \geq \tilde{\beta}_d \norm{v}_{Y_h}
    \qquad \forall v_h \in Y_h,
    \]
\end{enumerate}
then for each $(f,g) \in X' \times Y'$ there exists a unique $(\sigma_h, u_h) \in X_h \times Y_h$
solution to the following discrete perturbed saddle-point problem
\begin{align*}
    a(\sigma_h,\tau_h) + b(\tau_h,u_h)&=\;f(\tau_h)\qquad\forall\tau\in X_h, \\
 b(\sigma_h,v_h)- c(u_h,v_h) &=\;g(v_h)\qquad\forall v_h\in Y_h.
\end{align*}
Moreover, the solution satisfies the stability bound  
\[\|(\sigma_h,u_h)\|_{X\times Y} \lesssim \|f\|_{X'}+\|g\|_{Y'},\]
where the hidden constant depends only on the $h$-independent constants $\|a\|,\|b\|,\tilde{\alpha}_d$, and $\tilde{\beta}_d$.
\end{theorem}
With the above result we are able to immediately present the well-posedness of the discrete Brinkman problem under a fixed temperature.
\begin{lemma}\label{lem:Brinkman-discrete-solve}
    Under assumptions \ref{A1}--\ref{A2}, for any fixed $\widetilde{T}_h \in \rY_h$ there exists a unique $(\bomega_h, \bu_h, p_h) \in \bZ_h \times \bV_{h} \times \rQ_h$ such that the discrete operator equation $\mathcal{A}(\bomega_h, \bu_h, p_h) = (0,F(\widetilde{T}_h),0)$ is satisfied, implicitly defined by the system,
    \begin{align*}
    a_1(\bomega_h,\bzeta_h) + b_{1}(\bzeta_h,\bu_h)&=\;0&\qquad\forall\bzeta_h\in \bZ_h   ,\nonumber \\
 b_1(\bomega_h,\bv_h)- a_2(\bu_h,\bv_h) + b_2(\bv_h,p_h) &=\;F(\widetilde{T}_h;\bv_h)&\qquad\forall \bv_h\in\bV_h, \\
b_2(\bu_h,q_h)&=\;0&\qquad\forall q_h\in \rQ_h.\nonumber
\end{align*}
Furthermore we have the continuity of the discrete solution mapping by 
\[
    \norm{\bomega_h}_{\bcurl_s,\Omega} + \norm{\bu_h}_{r,\vdiv,\Omega} + \norm{p_h}_{0,\Omega}
    \lesssim (2 + \sqrt{\mu'} + \frac{\mu}{\kappa}) \rho \beta \norm{\bg}_{\bL^\infty} \|\widetilde{T}_h\|_{1,\Omega},
\]
where the hidden constant is $h$-independent and depends only on $\hat{\beta}_1, \hat{\beta}_2$ and $C_{\Omega,r}$. 
\end{lemma}

\begin{proof}
First let us consider the reduced discrete problem \eqref{eq:Brinkman-reduced-disc} in the discrete kernel of $\bB_2$. 
Take any $\bv_h$ in the kernel of the transpose operator $\bB_1^*$ restricted to $\bV_{0,h}$, then it necessarily satisfies 
\[
\int_\Omega \bcurl \bzeta_h \cdot \bv_h = 0 \qquad 
\forall \bv_h \in \ker(\bB_1^*|_{V_{0,h}}) \quad
\forall \bzeta_h \in \bZ_h. 
\]
As $\bZ_h \subset \bH_\star(\bcurl_s,\Omega)$ is conforming, then from \eqref{eq:char-2}
we have that $\ker(\bB_1^*|_{V_{0,h}}) = \{0\}$ is trivial.
Furthermore regarding conformity, the symmetry and positive semi-definiteness of $a_1,a_2$ are inherited.
From $\bZ_{0,h} \subset \bZ_0$ we have by \eqref{eq:a1-coer} that 
\begin{equation*}
    \sup_{\bzeta_h \in \bZ_{0,h} \setminus \{\cero\}} 
    \frac{\norm{a_1(\bomega_h, \bzeta_h)}}{\norm{\bomega_h}_{\bcurl_s, \Omega}}
    \geq \frac{\norm{a_1(\bomega_h, \bomega_h)}}{\norm{\bomega_h}_{\bcurl_s, \Omega}}
    = \norm{\bomega_h}_{\bcurl_s, \Omega}
    \qquad \forall \bomega_h \in \bZ_{0,h}. 
\end{equation*}
Thus in combination with the assumption \ref{A1},
we are able to apply \eqref{th:discrete-abstract} to yield a unique solution 
$(\bomega_h, \bu_h) \in \bZ_h \times \bV_{0,h}$ to \eqref{eq:Brinkman-reduced-disc} with continuity 
\begin{equation*}
    \norm{\bomega_h}_{\bcurl_s, \Omega} + \norm{\bu_h}_{r,\vdiv,\Omega} 
    \leq C_{s_1,d} \rho \beta \norm{\bg}_{\bL^\infty(\Omega)} \|\widetilde{T}_h\|_{1,\Omega}. 
\end{equation*}
Here $C_{s_1,d}$ is an $h$-independent constant, depending only on the volume constant $C_{\Omega,r}$ and $\hat{\beta}_1$.
Following a similar argument of \eqref{eq:G-problem-equiv}, we consider the restriction $\mathcal{G}_h$ of $\mathcal{G}$ to the finite-dimensional subspace $\bV_h$: 
\begin{equation*}
\mathcal{G}_h\bv_h := F(\widetilde{T}_h;\bv_h) - b_1(\bomega_h,\bv_h) + a_2(\bu_h,\bv_h) \qquad \forall \bv_h \in \bV_h.
\end{equation*}
Then, from conformity and in particular $\bV_{0,h} \subset \bV_{0}$, we have $\mathcal{G}_h |_{\bV_{0,h}} \equiv 0$ and hence there exists some $p_h \in \rQ_h$ such that $\bB_2^* p_h = \mathcal{G}_h$. Utilising assumption \ref{A1} affords uniqueness of $p_h$ by similar means to \eqref{eq:G-problem-equiv}, and by employing  assumption \ref{A2} we find the bound 
\begin{align*}
\norm{p_h}_{0,\Omega} \lesssim 
    \frac{1}{\hat{\beta}_2} \bigl(1 + \sqrt{\mu'} {+ \frac{\mu}{\kappa}}\bigr) \rho \beta \|\bg\|_{\bL^\infty(\Omega)} \|\widetilde{T}_h\|_{1,\Omega}.
\end{align*}
\end{proof}

Establishing the unique solvability of the discrete decoupled thermal problem is much the same due to the conformity, where the primary property inherited is the coercivity of $a_3(\cdot,\cdot)$.

\begin{lemma}\label{lem:energy-discrete}
 For a fixed $\tilde{\bu}_h\in\bV_{0,h}$,
 there exists a unique $T_h\in \rY_h$ such that 
\begin{equation}\label{eq:decoupled-energy-disc}
    a_3(T_h,S_h)+c_1(\tilde{\bu}_h;T_h,S_h)-c_2(\tilde{\bu}_h,\tilde{\bu}_h;S_h) 
    =\; G (S_h)\qquad \forall S_h\in \rY_h.
\end{equation}
Furthermore we have continuous dependence on data by,
\begin{equation}\label{eq:energy-cont-dep-discrete}
{\|T_h\|_{1,\Omega} \leq C_{\text{Sob}} \max\{\sigma_0^{-1},\alpha^{-1}\}}
\biggl[\|g\|_{0,\Omega} + \frac{\mu}{c'\rho\kappa} \|\tilde{\bu}_h\|^2_{r,\vdiv,\Omega}
\biggr].
\end{equation}
\end{lemma}

\begin{proof}
    Mimicking the proof from the continuous case, from $\bV_{0,h} \subset \bV_{0}$ we have that $c_1(\tilde{\bu}_h, \cdot, \cdot)|_{Y_h} \equiv 0$. Combining this with the inherited coercivity of $a_3(\cdot,\cdot)$ restricted to $\rY_h \times \rY_h$ and applying Lax--Milgram, we have a unique $T_h \in \rY_h$ solving \eqref{eq:decoupled-energy-disc}. Furthermore, we have the continuous dependence on data by 
    \begin{equation*}
        \|T_h\|_{1,\Omega} 
        \leq \max\{\sigma_0^{-1},\alpha^{-1}\} \norm{G + c_2(\tilde{\bu}_h, \tilde{\bu}_h; \cdot)}_{(Y_h)'}
        \leq C_{\text{Sob}} \max\{\sigma_0^{-1},\alpha^{-1}\}
\biggl[\|g\|_{0,\Omega} + \frac{\mu}{c'\rho\kappa} \|\tilde{\bu}_h\|^2_{r,\vdiv,\Omega}
\biggr],
    \end{equation*}
    where $C_{\text{Sob}}$ appears from the continuity of $c_2(\cdot,\cdot;\cdot)$ as seen in \eqref{eq:c1-c2-bound}.
\end{proof}

To conclude upon the discrete solvability we define the discrete solution operators,
\begin{alignat*}{2}
\mathcal{S}_{1,h}: &\ \rY_h &&\to \bZ_h \times \bV_{0,h},\\
                   &\ \tilde{T}_h &&\mapsto 
                    ({S}_{11,h} (\tilde{T}_h),{S}_{12,h} (\tilde{T}_h)) := (\bomega_h, \bu_h),
\end{alignat*}
where $(\bomega_h, \bu_h)$ is the unique solution to the discrete reduced problem \eqref{eq:Brinkman-reduced-disc} due to Lemma \ref{lem:Brinkman-discrete-solve}. Further define,
\begin{equation*}
    \mathcal{S}_{2,h} : \bV_{0,h} \to \rY_h, \quad \tilde{\bu}_h \mapsto \mathcal{S}_{2,h} (\tilde{\bu}_h) := T_h,
\end{equation*}
where $T_h$ is the unique solution to the discrete decoupled thermal problem \eqref{eq:energy-cont-dep-discrete}. To re-couple the discrete Brinkman and thermal problems, we consider the map $\mathcal{F}_h: \rY_h \to \rY_h$ given by $\mathcal{F}_h = \mathcal{S}_{2,h} \circ \mathcal{S}_{12,h}$. Then, one may observe that the solvability of the discrete nonlinear problem \eqref{eq:galerkin} is equivalent to the following: 
\begin{equation}\label{eq:discrete-fixed-point}
\text{Find $T_h \in \rY_h$ such that $\mathcal{F}_h(T_h) = T_h$.}
\end{equation}
Let us further define the discrete $h$-independent constants 
\begin{equation}
C_{1,d} := C_{s_2,d} \max\{\sigma^{-1}_0, \alpha^{-1}\} \norm{g}_{0,\Omega}
\,\,\, \text{and} \,\,\,
C_{2,d} := C^2_{s_1} C_{s_2} \max\{\sigma^{-1}_0, \alpha^{-1}\} 
    (\rho \beta \norm{\bg}_{\bL^\infty(\Omega)})^2 \frac{\mu}{c' \rho \kappa},
\end{equation}
and take the discrete analogue to \eqref{lem:F-map-b} by defining constants,
\begin{equation*}
x_{1,d} := \frac{1 - \sqrt{1 - 4C_{1,d}C_{2,d}}}{2C_{2,d}}
\qquad \text{and} \qquad
x_{2,d} := \frac{1 + \sqrt{1 - 4C_{1,d}C_{2,d}}}{2C_{2,d}},
\end{equation*}
which will be real under the appropriate assumptions.

\begin{theorem}\label{discrete-small-r}
Assume that $C_{1,d} C_{2,d} < 1/4$. Then, for a radius $R_d$ subject to both $x_{1,d} \leq R_d \leq x_{2,d}$ and,
\begin{equation}\label{eq:discrete-radius-assumption}
    R_d C_{s_1,d} C_{s_2,d} \rho \beta \norm{\bg}_{\bL^{\infty}(\Omega)} 
    \max\{\sigma_0^{-1}, \alpha^{-1}\} (1 + 2C_{s_1,d} \rho \beta \norm{\bg}_{\bL^{\infty}(\Omega)})
    < 1,
\end{equation}
there exists a unique solution to the discrete fixed-point problem \eqref{eq:discrete-fixed-point}, 
and hence equivalently showing the discrete well-posedness of the nonlinear problem \eqref{eq:galerkin}.
\end{theorem}

\begin{proof}
    Under the construction of discrete constants $C_{1,d}, C_{2,d}$ and suppositions regarding these constants, we may directly apply the same arguments as was used in the continuous setting. 
    Namely, by taking the following closed ball in the discrete space $\rY_h$:
    \begin{equation*}
    \rY^{R_d}_h := \{ S \in \rY_h : \norm{S}_{1,\Omega} \leq R_d \},
    \end{equation*}
    we have from $x_{1,d} \leq R_d \leq x_{2,d}$ and \eqref{lem:F-map-b} that 
    $\mathcal{F}_h(\rY^{R_d}_h) \subseteq \rY^{R_d}_h$. Similarly we have that $\mathcal{F}_h$
    is Lipschitz continuous by noting first from \eqref{lem:Brinkman-discrete-solve} that for any $T_{1,h}, T_{2,h} \in \rY_h$,
    \begin{align*}
    \norm{\mathcal{S}_{1,h}(T_{1,h}) - \mathcal{S}_{1,h}(T_{2,h})}
    &\leq \norm{\bomega_{1,h} - \bomega_{2,h}}_{\bcurl_s,\Omega} +
        \norm{\bu_{1,h} - \bu_{2,h}}_{r,\vdiv, \Omega}\\
    &\leq C_{s_1,d} \rho \beta \norm{\bg}_{\bL^\infty(\Omega)} \norm{T_{1,h} - T_{2,h}}_{1,\Omega}.
    \end{align*}
    Secondly, from \eqref{lem:energy-discrete} and conformity of the spaces, we have a similar estimate for $\bu_{1,h}, \bu_{2,h} \in \bV_{0,h}$,
    \begin{align*}
    &\norm{\mathcal{S}_{2,h}(\bu_{1,h}) - \mathcal{S}_{2,h}(\bu_{2,h})}_{1,\Omega}\\
    &\quad \leq C_{s_2,d} \max\{\sigma_0^{-1}, \alpha^{-1}\} 
    (\norm{\mathcal{S}_{2,h}(\bu_{2,h})}_{1,\Omega} + 
        \norm{\bu_{1,h}}_{r,\vdiv,\Omega} + \norm{\bu_{2,h}}_{r,\vdiv,\Omega})
    \norm{\bu_{1,h} - \bu_{2,h}}_{r,\vdiv,\Omega}.
    \end{align*}
    To conclude we consider then $T_{1,h}, T_{2,h} \in \rY^{R_d}_h$ which yields 
    \begin{align*}
  &  \norm{\mathcal{F}_h(T_{1,h}) - \mathcal{F}_h(T_{2,h})}_{1,\Omega} \\
&\qquad     \leq R_d C_{s_2,d} \max\{\sigma_0^{-1}, \alpha^{-1}\} 
        (1 + 2 C_{s_1,d} \rho \beta \norm{\bg}_{\bL^{\infty}(\Omega)})
        C_{s_1,d} \rho \beta \norm{\bg}_{\bL^{\infty}(\Omega)} \norm{T_{1,h} - T_{2,h}}_{1,\Omega},
    \end{align*}
    and we apply both the secondary radial assumption \eqref{eq:discrete-radius-assumption}, 
    as well as the Banach fixed-point theorem.
\end{proof}

\section{Quasi-optimality and convergence rates}\label{sec:error} This section is dedicated to the error analysis of \eqref{eq:galerkin}. First we use the generic assumptions \ref{A1}--\ref{A4} and unique solvability of continuous and discrete problems to obtain a C\'ea estimate. Then we choose specific finite element subspaces and show that they satisfy the generic assumptions, and provide concrete error bounds. 
\subsection{C\'ea estimate}
\begin{lemma}\label{cea-estimate}
    Suppose the conditions of \eqref{eq:small-r} and \eqref{discrete-small-r} are met such that $T,T_h$ are the unique fixed-point solutions   to the continuous and discrete problem, respectively.    Further consider the following assumption on constants
    \begin{equation}\label{eq:assumption-cea}
    (1 + \max(C_{s_1}, \beta_2^{-1}) C_{\Omega,r} \rho \beta \norm{\bg}_{\bL^\infty(\Omega)})
    \, \min(\sigma_0^{-1}, \alpha^{-1}) \, \mathcal{C}_1 < \frac12,
    \end{equation}
    where we set $\mathcal{C}_1$ to be,
    \begin{equation}\label{def:calC1}
    \mathcal{C}_1 := (1 + \frac{\mu}{\kappa c'} \beta \norm{\bg}_{\bL^\infty(\Omega)}) C_{s_1} R 
    + \frac{\mu}{\kappa c'} \beta \norm{\bg}_{\bL^\infty(\Omega)} C_{s_1,d} R_d.
    \end{equation}
    Then the corresponding solutions $(\bomega,\bu,p,T)$ and $(\bomega_h,\bu_h,p_h,T_h)$ are quasi-optimal, such that there exists an
    $h$-independent constant for which,
    \begin{align*}
        &\norm{\bomega - \bomega_h}_{\bcurl_s, \Omega} + \norm{\bu - \bu_h}_{r,\vdiv,\Omega} + \norm{p - p_h}_{0,\Omega} + \norm{T - T_h}_{1,\Omega} \\
        & \qquad \lesssim \dist(\bomega, \bZ_h) + \dist(\bu, \bV_h) + \dist(p, \rQ_h) + \dist(T, \rY_h). 
    \end{align*}
\end{lemma}

We briefly remark that given $\mathcal{C}_1$'s dependence on $R,R_d$, the assumption \eqref{eq:assumption-cea} will be dictated by the ratio of $\sigma_0$ and $\alpha$. This follows
from the lower bounds on $R,R_d$ by $x_1, x_{1,d}$ respectively.

\begin{proof}
Let us 
recall that the Brinkman operator $\mathcal{A}:\mathfrak{U} \to \mathfrak{U}'$ 
(with $\mathfrak{U}$ defined as in \eqref{def:frak}) is a continuous isomorphism by Lemma~\ref{lem:Brinkman}.  In particular for the fixed-point solution $(T,(\bomega,\bu,p)) \in \rH^1_\star(\Omega) \times \mathfrak{U}$ we have,
\begin{equation*}
\mathcal{A}((\bomega,\bu,p), (\bzeta, \bv, q)) = (\cero,F(T;v),0) \qquad \forall (\bzeta, \bv, q) \in \mathfrak{U},
\end{equation*}
and by $\bu \in \ker(\bB_2)$, the continuous dependence on data \eqref{eq:wu} of the reduced problem, and the bounded-below property of $\bB_2^*$, it holds that $\mathcal{A}$ is bounded below by,
\begin{equation*}
\norm{(\tilde{\bomega}, \tilde{\bu}, \tilde{p})}_{\mathfrak{U}} \leq \max(C_{s_1}, \beta_2^{-1}) \norm{\mathcal{A}((\tilde{\bomega}, \tilde{\bu}, \tilde{p}))}_{\mathfrak{U}'} 
\qquad \forall (\tilde{\bomega}, \tilde{\bu}, \tilde{p}) \in \mathfrak{U}.
\end{equation*}
Similarly as \eqref{def:frak}, we denote the discrete counterpart by $\mathfrak{U}_h := \bZ_h \times \bV_h \times \rQ_h$ and the discrete restriction $\mathcal{A}_h: \mathfrak{U}_h \to \mathfrak{U}_h'$,
which is another continuous isomorphism by Lemma~\ref{lem:Brinkman-discrete-solve} with fixed point $(T_h, (\bomega_h,\bu_h,p_h)) \in \rY_h \times \mathfrak{U}_h$ satisfying:
\begin{equation*}
\mathcal{A}_h((\bomega_h,\bu_h,p_h), (\bzeta_h, \bv_h, q_h)) = (\cero,F(T_h;\bv_h),0) \qquad \forall (\bzeta_h, \bv_h, q_h) \in \mathfrak{U}_h.
\end{equation*}
Furthermore $\mathcal{A}_h$ has an $h$-independent bounded-below property afforded by assumptions \ref{A1}--\ref{A4}, given as follows
\begin{equation*}
\norm{(\tilde{\bomega}_h, \tilde{\bu}_h, \tilde{p}_h)}_{\mathfrak{U}} 
\leq \max(C_{s_1,d}, \hat{\beta}_2^{-1}) \norm{\mathcal{A}_h(\tilde{\bomega}_h, \tilde{\bu}_h, \tilde{p}_h)}_{\mathfrak{U}_h'} 
\qquad \forall (\tilde{\bomega}_h, \tilde{\bu}_h, \tilde{p}_h) \in \mathfrak{U}_h.
\end{equation*}
As stated, $\mathcal{A}_h = \mathcal{A}|_{\mathfrak{U_h}}$ and $F_h = F|_{V_h}$ are simply restrictions so in constructing the errors,
\begin{equation*}
\be_{\bomega} := \bomega - \bomega_h \qquad
\be_{\bu}     := \bu     - \bu_h     \qquad
e_{p}       := p       - p_h       \qquad
e_{T}       := T       - T_h
\end{equation*}
one finds the following, recalling $\widetilde{F}$ denotes the linear part of the affine map $F: \rH^1_\star(\Omega) \to \bH^r_\star(\vdiv,\Omega)'$:
\begin{equation*}
\mathcal{A}((\be_{\bomega}, \be_{\bu}, e_{p}), \mathfrak{u}_h) = (\cero, \widetilde{F}(e_{T}; \mathfrak{u}_h), 0)
\qquad \forall \mathfrak{u}_h \in \mathfrak{U}_h.
\end{equation*}
If we let $(\tilde{\bomega}_h, \tilde{\bu}_h, \tilde{p}_h) \in \mathfrak{U}_h$ be arbitrary then
we may decompose each error in the usual way; for example in the case of $\be_{\bomega}$, as follows 
\begin{equation*}
\bchi_{\bomega}   := \bomega - \tilde{\bomega}_h \qquad
\bchi^h_{\bomega} := \tilde{\bomega}_h - \bomega_h \qquad \text{such that} \qquad
\be_{\bomega} = \bchi_{\bomega} + \bchi^h_{\bomega},
\end{equation*}
and analogously for each remaining individual error. From the bounded-below property of $\mathcal{A}_h$ and the continuity of $\widetilde{F},\mathcal{A}$, it follows that 
\begin{align*}
\norm{(\bchi^h_{\bomega},\bchi^h_{\bu},\chi^h_{p})}_{\mathfrak{U}}
&\leq \max(C_{s_1}, \beta_2^{-1}) \norm{\mathcal{A}((\be_{\bomega} - \bchi_{\bomega}, \be_{\bu} - \bchi_{\bu}, e_{p} - \chi_{p}))}_{\mathfrak{U}_h'}\\
&\leq \max(C_{s_1}, \beta_2^{-1}) (\Vert(0,\tilde{F}(e_T), 0)\Vert_{\mathfrak{U}_h'} + \norm{\mathcal{A}(\bchi_{\bomega}, \bchi_{\bu}, \chi_{p}))}_{\mathfrak{U}_h'} )\\
&\leq \max(C_{s_1}, \beta_2^{-1}) (C_{\Omega,r} \rho \beta \norm{\bg}_{\bL^\infty(\Omega)} \norm{e_T}_{1,\Omega} 
    + \norm{\mathcal{A}} \norm{(\bchi_{\bomega}, \bchi_{\bu}, \chi_{p}))}_{\mathfrak{U}} ),
\end{align*}
to which we may conclude the following partial result:
\begin{align*}
&\norm{(\be_{\bomega},\be_{\bu},e_{p})}_{\mathfrak{U}} \\
&\quad \leq  (1 + \max(C_{s_1}, \beta_2^{-1})\norm{\mathcal{A}}) \norm{(\bchi_{\bomega}, \bchi_{\bu}, \chi_{p}))}_{\mathfrak{U}} 
    + \max(C_{s_1}, \beta_2^{-1}) C_{\Omega,r} \rho \beta \norm{\bg}_{\bL^\infty(\Omega)} \norm{e_T}_{1,\Omega}.
\end{align*}
On the other hand, regarding the thermal problem one may view \eqref{eq:decoupled-energy} and \eqref{eq:decoupled-energy-disc} to deduce that 
\begin{equation*}
a_3(e_T, S_h) = c_1(\bu_h;T_h,S_h) - c_1(\bu;T, S_h) + c_2(\bu,\bu;S_h) - c_2(\bu_h,\bu_h;S_h) 
\qquad \forall S_h \in \rY_h.
\end{equation*}
To simplify the presentation of the
linear parts of the thermal-transport operator and the problem's corresponding functional,
we'll show two equalities regarding $c_1$ and $c_2$ respectively.
First, notice that by adding and subtracting $c_1(\bu; \tilde{T}_h,\cdot) - c_1(\bu; T_h, \cdot)$ we have
\begin{equation*}
c_1(\bu_h; T_h, \cdot) - c_1(\bu; T, \cdot) 
= -c_1(\bu - \bu_h; T_h, \cdot) - c_1(\bu; T - \tilde{T}_h, \cdot) - c_1(\bu; \tilde{T}_h - T_h, \cdot),
\end{equation*}
and so upon evaluation on $\chi^h_T$, recalling that both $\bu,\bu_h \in \bV_0$ from problem equivalence
and that the trilinear form $c_1(\cdot;\cdot,\cdot)$ is alternating over $\bV_0$ by (\ref{eq:c1-zero}), it follows that 
\begin{equation*}
c_1(\bu_h; T_h, \chi^h_T) - c_1(\bu; T, \chi^h_T)
= -c_1(\be_{\bu}; T_h, \chi^h_T) - c_1(\bu; \chi_T, \chi^h_T).
\end{equation*}
Secondly, we may do similar for $c_2(\cdot,\cdot;\cdot)$ by subtracting and adding the cross-term $c_2(\bu_h, \bu; \chi^h_T)$,
\begin{equation*}
c_2(\bu,\bu;\chi^h_T) - c_2(\bu_h, \bu_h; \chi^h_T)
= c_2(\bu - \bu_h,\bu; \chi^h_T) + c_2(\bu_h, \bu - \bu_h; \chi^h_T).
\end{equation*}
Now to estimate $\chi^h_T$ we shall make use of the $\rH^1_{\star}(\Omega)$-coercivity of $a_3(\cdot,\cdot)$ and the above equalities,
\begin{align*}
\min(\sigma_0, \alpha) \norm{\chi^h_T}^2
&  \leq a_3(\chi^h_T,\chi^h_T)\\
& = -c_1(\be_{\bu}, T_h,\chi^h_T) - c_1(\bu, \chi_T, \chi^h_T) - a_3(\chi_T,\chi^h_T)
    + c_2(\be_{\bu},\bu) + c_2(\bu_h, \be_{\bu})\\
&  \leq \Big(
    \big(
        \norm{c_1}\norm{T_h}_{1,\Omega} 
        + \norm{c_2}(\norm{\bu}_{r,\vdiv,\Omega} + \norm{\bu_h}_{r,\vdiv,\Omega})
    \big) \norm{\be_{\bu}}_{r,\vdiv,\Omega} \\
& \qquad \qquad + \big(
        \norm{c_1} \norm{\bu}_{r,\vdiv,\Omega} + \norm{a_3}
    \big) \norm{\chi_T}_{1,\Omega}
    \Big) \norm{\chi^h_T}_{1,\Omega}\\
&  := ( \,\, \tilde{\mathcal{C}}_1 \norm{\be_{\bu}}_{r,\vdiv,\Omega} + \tilde{\mathcal{C}}_2 \norm{\chi_T}_{1,\Omega} \,\, ) \norm{\chi^h_T}_{1,\Omega}.
\end{align*}
Utilising both the continuous dependence on data and the fact that each continuous and discrete fixed point $T,T_h$ resides within a respective ball around the origin, yields the following constant:
\begin{equation*}
\tilde{\mathcal{C}}_1 
\leq C_{\text{Sob}} \left(
    (1 + \frac{\mu}{\kappa c'} \beta \norm{\bg}_{\bL^\infty(\Omega)}) C_{s_1} R 
    + \frac{\mu}{\kappa c'} \beta \norm{\bg}_{\bL^\infty(\Omega)} C_{s_1,d} R_d
\right) = \mathcal{C}_1,
\end{equation*}
where $\mathcal{C}_1$ is as in \eqref{def:calC1}. 
Similarly, for $\mathcal{C}_2$ we obtain 
\begin{equation*}
\tilde{\mathcal{C}}_2
\leq C_{\text{Sob}} C_{s_1} \rho \beta \norm{\bg}_{\bL^\infty(\Omega)} R + \max(\sigma_0, \alpha)
=: \mathcal{C}_2.
\end{equation*}
This culminates in the following error estimate for the temperature field,
\begin{equation*}
\norm{e_{T}} \leq 
    \min(\sigma_0^{-1}, \alpha^{-1}) \, \mathcal{C}_1 \norm{\be_{\bu}}_{r,\vdiv,\Omega} 
        + (1 + \min(\sigma_0^{-1}, \alpha^{-1}) \, \mathcal{C}_2) \norm{\chi_T}_{1,\Omega},
\end{equation*}
for which we readily conclude with the following total error estimate
\begin{align*}
\norm{\be_{\bomega}}_{\bcurl_s,\Omega} &+ \norm{\be_{\bu}}_{r,\vdiv, \Omega} + \norm{e_{p}}_{0,\Omega} + \norm{e_T}_{1,\Omega}\\
&\leq (1 + \max(C_{s_1}, \beta_2^{-1}) \norm{\mathcal{A}}) \norm{(\bchi_{\bomega}, \bchi_{\bu}, \chi_{p})}\\
&\quad + (1 + \max(C_{s_1}, \beta_2^{-1}) C_{\Omega,r} \rho \beta \norm{\bg}_{\bL^\infty(\Omega)})
   (1 + \min(\sigma_0^{-1}, \alpha^{-1}) \mathcal{C}_2) \norm{\chi_T}_{1,\Omega}\\
&\quad + (1 + \max(C_{s_1}, \beta_2^{-1}) C_{\Omega,r} \rho \beta \norm{\bg}_{\bL^\infty(\Omega)})
   \min(\sigma_0^{-1}, \alpha^{-1}) \mathcal{C}_1 \norm{\be_{\bu}}_{r,\vdiv, \Omega}. 
\end{align*}
Given the assumption on parameters, subtracting the error term in velocity $\be_{\bu}$ and dividing throughout yields the claimed result, once the infimum of $((\tilde{\bomega}_h, \tilde{\bu}_h, \tilde{p}_h),\tilde{T}_h)$ in $\mathfrak{U}_h \times \rY_h$ is taken.
\end{proof}

\subsection{Specific finite element spaces}
By $\mathrm{P}_k(K)$ and $\mathbf{P}_k(K)$ we will denote scalar and vector polynomial spaces of degree up to $k$, defined locally on $K\in \cT_h$. In addition we denote by $\mathbf{RT}_k(K)=\mathbf{P}_k(K) \oplus \mathrm{P}_k(K)\bx$ the local Raviart--Thomas space and denote by $\mathbf{RT}_k(\cT_h)$ its global counterpart. 
From now on we use as discrete spaces  the $\bH(\bcurl_s,\Omega)$-conforming N\'ed\'elec elements of the first kind and order $k+1$ for vorticity (the local space denoted as $\mathbf{ND}_{k+1}(K)$), the $\bH^r(\vdiv,\Omega)$-conforming Raviart--Thomas elements of degree $k \geq 0$ for velocity approximation,  discontinuous and piecewise polynomials of degree $k$ for pressure, and continuous and piecewise polynomials of degree $k+1$ for temperature 
\begin{align}
\nonumber   \bZ_h &:=  \mathbf{ND}_{k+1}(\cT_h)\cap  \bH_\star(\bcurl_s,\Omega) = \{\bzeta_h\in {\bH_\star(\bcurl_s,\Omega)}: \bzeta_h|_K \in \mathbf{ND}_{k+1}(K),\quad \forall K\in \cT_h\},\\
\label{eq:fespaces} \bV_h &:= \mathbf{RT}_k(\cT_h)\cap \bH^r_\star(\vdiv,\Omega) = 
\{\bv_h \in {\bH^r_\star(\vdiv,\Omega)}: \bv_h|_K \in \mathbf{RT}_k(K),\quad \forall K\in \cT_h\},\\
\nonumber  \rQ_h &:= \mathrm{P}_k(\cT_h) =
\{q_h \in {\rL^2(\Omega)}: q_h|_K \in \mathrm{P}_k(K),\quad \forall K\in \cT_h\},\\
\nonumber  \rY_h &:= \mathrm{P}_{k+1}(\cT_h) \cap \mathrm{C}(\overline{\Omega})\cap \rH^1_\star(\Omega) = 
\{S_h \in \mathrm{C}(\overline{\Omega})\cap{\rH^1_\star(\Omega)}: S_h|_K \in \mathrm{P}_{k+1}(K),\quad \forall K\in \cT_h\}.
\end{align}

Let us denote by  $\cP_h:\rL^2(\Omega)\rightarrow \rQ_h$  the $\rL^2$-orthogonal projection into $\rQ_h$. In addition, let $\Gamma_h$ denote the set of facets on $\cT_h$ that lie on $\Gamma$ and denote by $\mathrm{P}_k(\Gamma_h)$ the space of piecewise polynomials of degree up to $k$ defined on each $e\in \Gamma_h$. We also denote by $\cP_h^\Gamma: \rL^1(\Gamma)\to \mathrm{P}_k(\Gamma_h)$ the orthogonal projector with respect to the $\rL^2(\Gamma)$-inner product. 
Consider for $t\in [1,\infty)$ the space 
\[ \bV^t = \{\bv \in \bH^t(\vdiv,\Omega): \bv \in \bW^{1,t}(K) \ \forall K\in \cT_h\}.\]
From \cite[Appendix A]{caucao23} we recall that the global Raviart--Thomas interpolator $\mathcal{I}_h^{\mathrm{RT}}:\bV_t\rightarrow \bV_h$ satisfies the Fortin-type property  
\begin{equation}\label{eq:commuting}
    \vdiv(\mathcal{I}_h^{\mathrm{RT}}(\bv)) = \cP_h(\vdiv\bv) \quad \forall \bv \in \bV_t.
\end{equation}
Moreover, we have that 
\begin{equation}
\label{eq:commuting-facet}
\mathcal{I}_h^{\mathrm{RT}}(\bv)\cdot \bn = \cP_h^\Gamma(\bv\cdot\bn) \quad \text{on} \ \Gamma \quad \forall \bv \in \bV_t,
\end{equation}
and we recall the following estimates
\begin{subequations}\label{eq:estimate-RT-Ph}
\begin{align}\label{estimate:RT}
    \|\mathcal{I}_h^{\mathrm{RT}}(\bv)\|_{\bL^{t}(\Omega)} &\lesssim \|\bv\|_{\bW^{1,t}(\Omega)} \quad \forall \bv \in \bW^{1,t}(\Omega),\\
    \|\cP_h(q)\|_{0,\Omega} &\lesssim \|q\|_{0,\Omega} \quad \forall q \in \rL^{2}(\Omega),\label{estimate:Lt}
\end{align}\end{subequations}
whose proof can be found in \cite{gatica22}.

\begin{remark}\label{lemma:uni-bound-RT}
Note that \eqref{estimate:RT} follows from the estimate  $\|\bv - \mathcal{I}_h^{\mathrm{RT}}(\bv)\|_{\bL^{t}(K)}\lesssim h_K^{l+1} |\bv|_{\bW^{1,t}(K)}$ for all $K\in \cT_h$ and $0\leq l\leq k$ (cf. \cite[Remark 4.2]{camano21}), using the condition $h_K\leq h\leq h_0$, where $h_0$ denotes the diameter of $\Omega$, and from the triangle inequality, we have
\[\ds \|\mathcal{I}_h^{\mathrm{RT}}(\bv)\|_{\bL^{t}(\Omega)}\leq \|\bv - \mathcal{I}_h^{\mathrm{RT}}(\bv)\|_{\bL^{t}(\Omega)} + \|\bv\|_{\bL^{t}(\Omega)} \leq c(h_0^{(l+1)/t} +1) \|\bv\|_{\bW^{1,t}(\Omega)},\]
with $c$ independent of $h$.
\end{remark}

It is important to remark that because of the interaction with the discrete space for vorticity, we will also require the space 
\[ \bV_{h,s} :=\mathbf{RT}_k(\cT_h)\cap \bH^s_\star(\vdiv,\Omega),\]
along with a dedicated version of the canonical Raviart--Thomas projection $\mathcal{I}_h^{\mathrm{RT}{,s}}:\bV_t\rightarrow \bV_{h,s}$, satisfying analogous properties as above. 
Similarly as for the Raviart--Thomas interpolator, we have in particular the following properties for the global N\'ed\'elec interpolator (see \cite[Lemma 16.8 \& {proof of Theorem 16.12}]{ern21}) $\mathcal{I}^{\mathrm{N}}_h:\bZ^t\to \bZ_h$ 
\begin{subequations}
\begin{align}
\label{eq:commuting-nedelec}
\bcurl(\mathcal{I}_h^\mathrm{N}(\bzeta)) &= \mathcal{I}_h^{\mathrm{RT},s}(\bcurl\bzeta) ,\\
  \|\mathcal{I}_h^{\mathrm{N}}(\bzeta)\|_{\bL^{2}(K)} &\lesssim \|\bzeta\|_{\bL^{2}(N(K))} + h\|\bcurl\bzeta\|_{\bL^t(N(K))} \lesssim \|\bzeta\|_{\bcurl_t,N(K)} ,\label{estimate:nedelec} 
\end{align}
\end{subequations}
for all $\bzeta \in \bZ^t$, where $N(K)$ denotes the neighbourhood of an element $K\in \cT_h$, and for $t\in (1,\infty)$ and $\delta t>2$, the domain space $\bZ^t$ is defined as (see 
\cite[eq. (16.10)]{ern21})
\[ \bZ^t : = \{ 
\bzeta \in \bH(\bcurl_t,\Omega): \bzeta \in \bW^{\delta,t}(K)
 \quad \forall K\in \cT_h\}.\]

\begin{remark}
The proof of \eqref{estimate:nedelec} follows the same arguments as \eqref{estimate:RT} (see Remark~\ref{lemma:uni-bound-RT}), utilising the properties of N\'ed\'elec interpolator.
\end{remark}
Note that, in essence, the only requirement on the domain spaces $\bV^t,\bZ^t$ is sufficient regularity so that the Raviart--Thomas and N\'ed\'elec degrees of freedom are linear functionals in their respective discrete spaces.  

Let us recall, from, e.g.,  \cite[Chapters 16, 17 and 22]{ern21} and \cite[Appendix A]{caucao23} the following approximation properties of the 
finite element subspaces \eqref{eq:fespaces}, which are obtained using the Bramble--Hilbert Lemma and scaling arguments applied to each contribution in the specific norms.   
Assume that 
\begin{enumerate}
\item $\bzeta \in \bH_\star(\bcurl_{s},\Omega)\cap  \bH^{m}(\Omega)$ with $\bcurl\bzeta \in \bW^{m,s}(\Omega)$,
\item $\bv \in \bH^{r}_\star(\vdiv,\Omega) \cap \bW^{m,r}(\Omega)$ with $\vdiv\bv \in \rH^m(\Omega)$,
\item $q\in  {\rH}^{m}(\Omega)$,
\item $S \in \rH^{m+1}(\Omega)$,
\end{enumerate} 
for some
$m\in(1/2,k+1]$. Then, there exists $C>0$ 
independent of $h$, such that
\begin{subequations}
\begin{align}
\|\bzeta-\mathcal{I}_h^{\mathrm{N}}\bzeta\|_{\bcurl_s,\Omega} &\leq C h^{m}(\|\bzeta\|_{m,\Omega} + \|\bcurl\bzeta\|_{\bW^{m,s}(\Omega)}),\label{Ap1}\\
\|\bv-\mathcal{I}_h^{\mathrm{RT}}\bv \|_{r,\vdiv,\Omega}&\leq C h^{m}(\|\bv\|_{\bW^{m,r}(\Omega)} + \|\vdiv\bv\|_{m,\Omega}), \label{Ap2}\\
\Vert q -\cP_h q \Vert_{0,\Omega} &\leq C h^{m}|q |_{m,\Omega},\label{Ap3}\\
\|S-\mathcal{I}_hS \|_{1,\Omega} &\leq  C h^{m}\|S\|_{1+m,\Omega},\label{Ap4}
\end{align}
\end{subequations}
where  $\mathcal{I}_h$ denotes  the 
 Lagrange interpolator.

\subsection{Verification of general hypotheses}
\paragraph{\underline{Inf-sup condition \ref{A1}}.} Consider $\bv_h\in \bV_{0,h}$. We can readily construct $\mathcal{D}_s[\mathcal{J}_r(\bv_h)] \in \bL^s(\Omega)$. By construction this function has zero divergence. Then, guided by the arguments used in the first part of the proof of Lemma~\ref{lem:inf-sup} we can find the unique $\bz \in \bH_\star(\bcurl_s,\Omega)$ such that 
\begin{equation}\label{baux03}\bcurl \bz = \mathcal{D}_s[\mathcal{J}_r(\bv_h)] , \quad \vdiv \bz = 0,\end{equation}
and 
\begin{equation}\label{baux04}
\|\bz\|_{\bcurl_s,\Omega} \leq \|\bz\|_{\bW^{1,s}(\Omega)}\lesssim \| \mathcal{D}_s[\mathcal{J}_r(\bv_h)] \|_{\bL^s(\Omega)},\end{equation} 
where the hidden constant only depends on the domain. 

We now define $\tilde{\bz}_h$ as the N\'ed\'elec interpolant of the $\bz$ found above
\[ \tilde{\bz}_h : = \mathcal{I}^{\mathrm{N}}_h\bz \in \bZ_h,\]
and use the commutativity of the N\'ed\'elec and Raviart--Thomas interpolation \eqref{eq:commuting-nedelec} to write 
\begin{equation}\label{baux06}\bcurl \tilde{\bz}_h = \bcurl \mathcal{I}^{\mathrm{N}}_h\bz =  \mathcal{I}^{\mathrm{RT}{,s}}_h(\bcurl \bz) = \mathcal{I}^{\mathrm{RT}{,s}}_h(\mathcal{D}_s[\mathcal{J}_r(\bv_h)] ) = {\mathcal{D}_s[\mathcal{J}_r(\bv_h)] },\end{equation}
where we have employed \eqref{baux03} and the fact that $\mathcal{D}_s[\mathcal{J}_r(\bv_h)]$ is in $\bV^t$ {as well as in $\bV_{h,s}$}.

Note also that since $\bz \in \bL^2(\Omega)$ (cf. proof of Lemma~\ref{lem:inf-sup}), then $\tilde{\bz}_h$ is also in $\bL^2(\Omega)$. Then, using triangle inequality, the bound \eqref{baux04}, the continuity of the N\'ed\'elec interpolator  \eqref{estimate:nedelec}, and \eqref{baux06}, we can write 
\begin{equation}\label{baux07} \|\tilde{\bz}_h\|_{\bcurl_s,\Omega} 
\leq \| \mathcal{I}^{\mathrm{N}}_h \bz\|_{\bL^2(\Omega)} + \| \bcurl \mathcal{I}^{\mathrm{N}}_h\bz 
\|_{\bL^s(\Omega)} \lesssim \| \mathcal{D}_s[\mathcal{J}_r(\bv_h)]  \|_{\bL^s(\Omega)}. \end{equation}

Therefore, using the definition of the bilinear form $b_1(\cdot,\cdot)$, relation  \eqref{baux06}, and the bound \eqref{baux07}, we can derive the inf-sup condition as follows 
\begin{align*}  \sup_{\bzeta_h \in \bZ_h\setminus\{\cero\}} \frac{b_1(\bzeta_h,\bv_h)}{\|\bzeta_h\|_{\bcurl_s,\Omega}} & \geq \frac{\int_\Omega \bcurl\tilde{\bz}_h\cdot \bv_h}{\|\tilde{\bz}_h\|_{\bcurl_s,\Omega}} = \frac{\int_\Omega \mathcal{I}^{\mathrm{RT}{,s}}_h(\cD_s[\mathcal{J}_r(\bv_h)])  \cdot \bv_h}{\|\tilde{\bz}_h\|_{\bcurl_s,\Omega}} \\
&  =  \frac{\int_\Omega {\cD_s[\mathcal{J}_r(\bv_h)]}  \cdot \bv_h}{\|\tilde{\bz}_h\|_{\bcurl_s,\Omega}} =
\frac{\| {\cD_s[\mathcal{J}_r(\bv_h)]} \|_{\bL^{s}(\Omega)} \|\bv_h\|_{\bL^r(\Omega)}}{\|\tilde{\bz}_h\|_{\bcurl_s,\Omega}} 
\\
& \gtrsim  \|\bv_h\|_{\bL^r(\Omega)} \qquad \forall \bv_h \in \bV_{0,h},\end{align*}
where we have also used \eqref{eq-lem-S-2}. The discrete inf-sup constant is independent of $h$, but depends on the continuous dependence on data of the auxiliary problem and on the continuity bound of the N\'ed\'elec interpolator. 

\paragraph{\underline{Inf-sup condition \ref{A2}}.} We start by considering $q_h\in\rQ_h$. Then, 
mimicking the proof of the continuous inf-sup condition for $b_2(\cdot,\cdot)$ in the second part of Lemma~\ref{lem:inf-sup}, we are able to construct 
$\tilde{\bv} \in \bH^1(\Omega)$ satisfying 
\[ \vdiv{\tilde{\bv}} = q_h \qan \|\tilde{\bv}\|_{1,\Omega} \lesssim \|q_h\|_{0,\Omega}.\]
Then we take $\tilde{\bv}_h = \mathcal{I}_h^{\mathrm{RT}}\tilde{\bv}$ and use the properties of the Raviart--Thomas interpolator \eqref{eq:commuting}, to get 
\[ \vdiv \tilde{\bv}_h = \vdiv \mathcal{I}_h^{\mathrm{RT}}\tilde{\bv} = \cP_h(\vdiv\tilde{\bv}) = \cP_h(q_h) \quad \text{in } \ \Omega,\]
and since from \eqref{eq:commuting-facet} we also have that 
\[ \tilde{\bv}_h \cdot \bn = \mathcal{I}_h^{\mathrm{RT}}(\tilde{\bv}) \cdot \bn = \cP_h^\Gamma(\tilde{\bv}\cdot\bn) = 0 \quad \text{on } \ \Gamma ,\]
then we can verify that $\tilde{\bv}_h \in \bV_h$. Next we use the stability of the Raviart--Thomas interpolator \eqref{estimate:RT} and of the $\rL^2$-projection \eqref{estimate:Lt} to obtain 
\begin{align*}
    \|\tilde{\bv}_h\|_{\bL^r(\Omega)} &=  \|\mathcal{I}_h^{\mathrm{RT}}(\tilde{\bv}) \|_{\bL^r(\Omega)} \lesssim 
    \|\tilde{\bv}\|_{\bH^1(\Omega)} \lesssim \|q_h\|_{0,\Omega},\\
    \|\vdiv\tilde{\bv}_h\|_{0,\Omega} &=  \|\cP_h (q_h)\|_{0,\Omega} \lesssim \|q_h\|_{0,\Omega},
\end{align*}
and thus we can assert that $\|\tilde{\bv}_h\|_{r,\vdiv,\Omega} \lesssim \|q_h\|_{0,\Omega}$. 

In this way, from the previous relations it follows  that 
\begin{align*}
   \sup_{\bv_h \in \bV_h\setminus\{\cero\}} \frac{b_2(\bv_h,q_h)}{\|\bv_h\|_{r,\vdiv,\Omega}} 
   & \geq 
   \frac{\int_\Omega\vdiv(\tilde{\bv}_h)\,q_h}{\|\tilde{\bv}_h\|_{r,\vdiv,\Omega}} = 
    \frac{\int_\Omega\cP_h(q_h)\,q_h}{\|\tilde{\bv}_h\|_{r,\vdiv,\Omega}} =  \frac{\|q_h\|^2_{0,\Omega}}{\|\tilde{\bv}_h\|_{r,\vdiv,\Omega}} 
    \geq \hat{\beta}_2 \|q_h\|_{0,\Omega} \quad \forall q_h\in \rQ_h,
\end{align*}
where the inf-sup constant depends on the continuous dependence on data of the continuous Stokes equations and on the boundedness constants of the Raviart--Thomas interpolation and $\rL^2$-projection \eqref{eq:estimate-RT-Ph}, and it is independent of $h$. 

\paragraph{\underline{Conditions \ref{A3}--\ref{A4}}.} These assumptions can be straightforwardly verified from the definition of the finite element spaces.  

\subsection{Convergence rates}
Given that the desired assumptions \ref{A1}--\ref{A4} which yield the conditional approximation quasi-optimality \eqref{cea-estimate} are satisfied, the following corollary is readily afforded by additionally bounding above by the interpolator estimates \eqref{Ap1}-\eqref{Ap4}.
\begin{lemma}
Suppose the conditions of \eqref{eq:small-r} and \eqref{discrete-small-r} are met such that $T,T_h$ are the unique fixed-point solutions to the continuous and discrete problem, respectively.
Additionally suppose that the assumption on constants \eqref{eq:assumption-cea} is met.
Then there exists an $h$-independent constant $C >0$ such that,
\begin{align*}
&\norm{\bomega - \bomega_h}_{\bcurl_s, \Omega} + \norm{\bu - \bu_h}_{r,\vdiv,\Omega} + \norm{p - p_h}_{0,\Omega} + \norm{T - T_h}_{1,\Omega} \\
&\quad  \leq C h^m (
\norm{\bomega}_{m,\Omega} + \norm{\bcurl \bomega}_{\bW^{m,s}(\Omega)} 
+ \norm{\bv}_{\bW^{m,r}(\Omega)} + \norm{\vdiv \bv}_{m,\Omega}
+ |p|_{m,\Omega}
+ \norm{S}_{1+m, \Omega}
).
\end{align*}
\end{lemma}

\section{Numerical examples}\label{sec:results}
The numerical implementation uses the open-source finite element library \texttt{Gridap} \cite{badia22}. In all cases the Newton--Raphson iterations are stopped once either the absolute or the relative $\ell^2$-norm of the residuals get below $10^{-8}$, and we use the direct method MUMPS for the tangent linear systems.

We proceed to validate the finite element method by  convergence verification, and then we apply the proposed formulation in the simulation of non-isothermal rotational flows.

\subsection{Accuracy tests} We consider a 2D and a 3D manufactured solutions computation with mixed boundary conditions. On the domain $\Omega = (0,2) \times(0,1)$ we consider closed-form solutions to the vorticity-based non-isothermal dissipative flow equations as follows 
\begin{gather*}
\bu(x,y) = \begin{pmatrix}
\cos(\pi x)\sin(\pi y)\\
-\sin(\pi x)\cos(\pi y)  
\end{pmatrix}, \quad \omega = \sqrt{\mu'}\curl \bu, \\ 
 p(x,y) = \frac12x^4-y^4,\quad 
T(x,y) = 1+\cos^2(\pi xy),
\end{gather*} 
and the external force and heat source terms, together with non-homogeneous essential boundary conditions are computed from these manufactured solutions. The part of the boundary $\Gamma$ where we impose the non-homogeneous counterpart of \eqref{bc:Gamma} is conformed by the segments $x=0$ and $y=0$, and  $\Sigma = \partial\Omega \setminus \Gamma$ by $x=2$ and $y=1$. 
For the 3D case we take  $\Omega = (0,1)\times (0,\frac12)\times(0,\frac12)$ together with the manufactured solutions 
\begin{gather*}
\bu(x,y,z) = \begin{pmatrix}
\sin^2(\pi x)\sin(\pi y)\sin(2\pi z)\\
\sin(\pi x)\sin^2(\pi y)\sin(2\pi z)\\ 
-[\sin(2\pi x)\sin(\pi y)+\sin(\pi x)\sin(2 \pi y)]\sin^2(\pi z)  
\end{pmatrix}, \quad \bomega = \sqrt{\mu'}\bcurl \bu, \\
\quad p(x,y,z) = \sin(\pi x)\cos(\pi y)\sin(\pi z),\quad 
T(x,y,z) = 1+ \sin^2(\pi x)\sin^2(\pi y)\sin^2(\pi z),
\end{gather*} 
and the boundary splitting has the faces $x=0$, $y=0$ and $z=0$ in $\Gamma$, and $x=1$, $y=\frac12$, $z=\frac12$ in $\Sigma$. 

\begin{table}[!t]
\setlength{\tabcolsep}{2.7pt}
\begin{center}
\begin{tabular}{|rccccccccccc|}
\hline
DoF  & $h$ &  $e_{\curl_s}(\omega)$  &  \texttt{rate}  &   $e_{r,\vdiv}(\bu)\!$ 
&   \texttt{rate}  &  $e_0(p)$  &  \texttt{rate} &  $e_1(T)$  &  \texttt{rate} & $\|\vdiv\bu_h\|_{\ell^\infty}$ &  \texttt{it} \\  
\hline
\multicolumn{12}{|c|}{Errors and convergence rates for $k = 0$}\\
\hline
   132 & 0.5000 & 8.87e+0 & $\star$ & 4.25e-01 & $\star$ & 7.81e-01 & $\star$ & 4.30e+0 & $\star$ & 8.88e-16 & 4\\
   486 & 0.2500 & 4.23e+0 & 1.07 & 2.07e-01 & 1.04 & 3.75e-01 & 1.06 & 2.22e+0 & 0.95 & 1.78e-15 & 4\\
  1866 & 0.1250 & 2.06e+0 & 1.04 & 1.02e-01 & 1.02 & 1.85e-01 & 1.02 & 1.16e+0 & 0.94 & 3.55e-15 & 4\\
  7314 & 0.0625 & 1.02e+0 & 1.02 & 5.09e-02 & 1.00 & 9.20e-02 & 1.01 & 5.69e-01 & 1.03 & 7.11e-15 & 4\\
 28962 & 0.0312 & 5.05e-01 & 1.01 & 2.54e-02 & 1.00 & 4.59e-02 & 1.00 & 2.81e-01 & 1.02 & 1.42e-14 &4 \\
 115266 & 0.0156 & 2.52e-01 & 1.00 & 1.27e-02 & 1.00 & 2.30e-02 & 1.00 & 1.40e-01 & 1.01 & 5.68e-14 &4 \\
\hline
\multicolumn{12}{|c|}{Errors and convergence rates for $k = 1$}\\
\hline
  422 & 0.5000 & 1.67e+0 & $\star$ & 9.08e-02 & $\star$ & 7.56e-02 & $\star$ & 1.66e+0 & $\star$ & 7.16e-14 & 4 \\
  1610 & 0.2500 & 4.29e-01 & 1.96 & 2.42e-02 & 1.91 & 1.78e-02 & 2.09 & 6.19e-01 & 1.42 & 1.46e-13 & 4 \\
  6290 & 0.1250 & 1.07e-01 & 2.01 & 6.13e-03 & 1.98 & 4.37e-03 & 2.02 & 1.50e-01 & 2.05 & 3.07e-13 & 4 \\
 24866 & 0.0625 & 2.65e-02 & 2.01 & 1.54e-03 & 2.00 & 1.09e-03 & 2.01 & 3.83e-02 & 1.97 & 6.37e-13 & 4 \\
 98882 & 0.0312 & 6.61e-03 & 2.00 & 3.85e-04 & 2.00 & 2.72e-04 & 2.00 & 9.64e-03 & 1.99 & 1.24e-12 & 4 \\
 394370 & 0.0156 & 1.65e-03 & 2.00 & 9.62e-05 & 2.00 & 6.80e-05 & 2.00 & 2.41e-03 & 2.00 & 2.54e-12 & 4\\
 \hline
\end{tabular}
\end{center}
\caption{Accuracy test in 2D. Error history (errors for each field variable in the corresponding norm on a sequence of successively refined grids, numerically computed  convergence rates, 
and discrete norm of the divergence of the approximate velocity) with $r=6$, $s=\frac65$, and for $\mathrm{P}_{k+1}-\mathbf{RT}_k-\mathrm{P}^{\mathrm{disc}}_k-\mathrm{P}_{k+1}$ elements with different polynomial degrees, and iteration count for the nonlinear Newton--Raphson solver. The symbol $\star$ indicates that no rate is computed at the initial coarse mesh refinement.} \label{table01}
\end{table}

\begin{table}[!t]
\setlength{\tabcolsep}{2.7pt}
\begin{center}
\begin{tabular}{|rccccccccccc|}
\hline
DoF  & $h$ &  $e_{\bcurl_s}(\bomega)$  &  \texttt{rate}  &   $e_{r,\vdiv}(\bu)\!$ 
&   \texttt{rate}  &  $e_0(p)$  &  \texttt{rate} &  $e_1(T)$  &  \texttt{rate} & $\|\vdiv\bu_h\|_{\ell^\infty}$ &  \texttt{it} \\  
\hline
\multicolumn{12}{|c|}{Errors and convergence rates for $k = 0$}\\
\hline
    91 & 0.8660 & 2.28e+1 & $\star$ & 4.10e-01 & $\star$ & 4.02e-01 & $\star$ & 2.64e+0 & $\star$ & 1.67e-16  & 4\\
   553 & 0.4330 & 1.62e+1 & 0.49 & 2.20e-01 & 0.90 & 1.42e-01 & 1.50 & 1.40e+0 & 0.91 & 2.66e-15  & 4\\
  3841 & 0.2165 & 8.86e+0 & 0.87 & 1.15e-01 & 0.93 & 5.29e-02 & 1.43 & 5.19e-01 & 1.44 & 6.00e-15 & 3\\ 
 28609 & 0.1083 & 4.56e+0 & 0.96 & 5.88e-02 & 0.97 & 1.73e-02 & 1.61 & 1.83e-01 & 1.50 & 2.00e-14 & 3 \\
220801 & 0.0541 & 2.30e+0 & 0.99 & 2.95e-02 & 0.99 & 6.61e-03 & 1.39 & 7.41e-02 & 1.31 & 5.70e-14 & 3\\
\hline
\multicolumn{12}{|c|}{Errors and convergence rates for $k = 1$}\\
\hline
  365 & 0.8660 & 1.20e+1 & $\star$ & 2.03e-01 & $\star$ & 2.29e-01 & $\star$ & 7.76e-01 & $\star$ & 4.05e-15  & 4\\
  2417 & 0.4330 & 4.96e+0 & 1.28 & 6.21e-02 & 1.71 & 6.44e-02 & 1.83 & 2.52e-01 & 1.62 & 2.11e-14  & 4\\
 17537 & 0.2165 & 1.39e+0 & 1.84 & 1.69e-02 & 1.88 & 1.13e-02 & 2.51 & 4.74e-02 & 2.41 & 7.76e-14  & 3\\
133505 & 0.1083 & 3.67e-01 & 1.92 & 4.38e-03 & 1.95 & 2.04e-03 & 2.47 & 1.10e-02 & 2.11 & 6.12e-13  & 3\\
1041665 & 0.0541 & 9.26e-02 & 1.96 & 1.11e-03 & 1.98 & 4.29e-04 & 2.25 & 2.72e-03 & 2.02 & 2.47e-12 & 3\\
 \hline
\end{tabular}
\end{center}
\caption{Accuracy test in 3D. Error history (errors for each field variable in the corresponding norm, with $r=6$, $s=\frac65$, on a sequence of successively refined grids, numerically computed convergence rates, 
and discrete norm of the divergence of the approximate velocity)   for $\mathbf{ND}_{k+1}-\mathbf{RT}_k-\mathrm{P}^{\mathrm{disc}}_k-\mathrm{P}_{k+1}$ elements with different polynomial degrees, and iteration count for the nonlinear Newton--Raphson solver. The symbol $\star$ indicates that no rate is computed at the initial coarse mesh refinement.} \label{table02}
\end{table}

In both 2D and 3D cases the remaining model  parameters are taken all as one 
 $\bg = (0,0,-1)^{\tt t}$ (and $\bg = (0,-1)^{\tt t}$ in 2D),  $\rho = \mu = \mu'=\kappa = c'=\beta = T_0=1$ (in their respective units). 
 Approximate solutions are computed on a sequence of $\texttt{n}_k^{\max}$ successively refined uniform  tetrahedral (triangular in 2D) meshes. Then we generate the error history at each refinement level, consisting of the error of each unknown 
 \begin{gather*} e_{\curl_s}(\bomega) = \|\bomega-\bomega_h\|_{\bcurl_s,\Omega}, \qquad e_{r,\vdiv}(\bu)= \|\bu-\bu_h\|_{r,\vdiv,\Omega}, \\ e_0(p) = \|p-p_h\|_{0,\Omega}, \qquad e_1(T) = \|T-T_h\|_{1,\Omega}, \end{gather*}
with $r=6$, $s=\frac65$, as well as the experimental convergence rate of the error decay as 
 \[ \texttt{rate} = \frac{\log(e_i(\cdot))- \log(e_{i+1}(\cdot))}{\log(h_i)- \log(h_{i+1})},\]
where we denote by $e_i$ the error associated with an approximation computed on the $i$-th mesh refinement level having a grid of meshsize $h_i$.

\begin{figure}[t!]
    \centering
\includegraphics[width=0.45\textwidth]{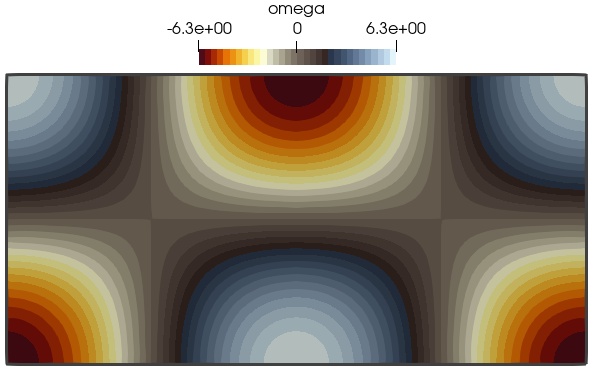}
\includegraphics[width=0.45\textwidth]{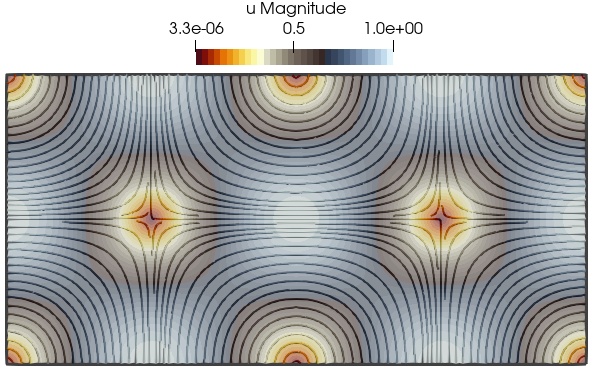}\\
\includegraphics[width=0.45\textwidth]{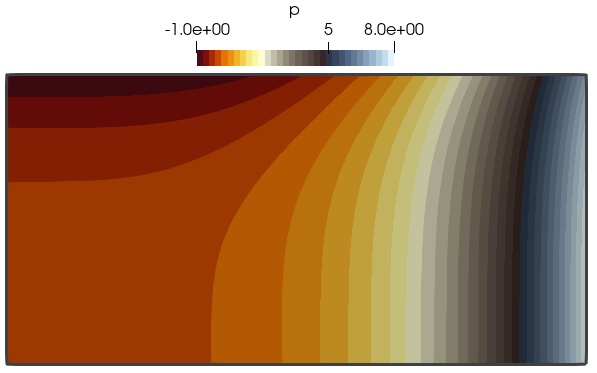}
\includegraphics[width=0.45\textwidth]{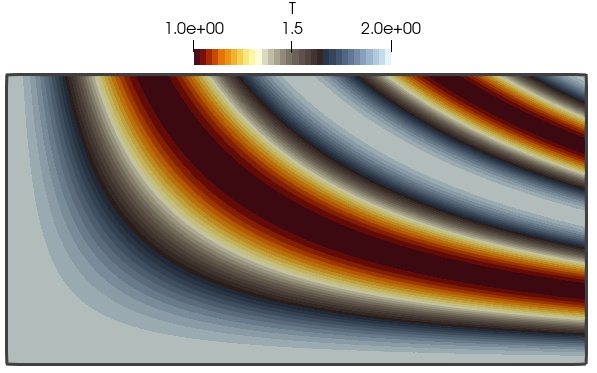}\\[2ex]
\includegraphics[width=0.45\textwidth]{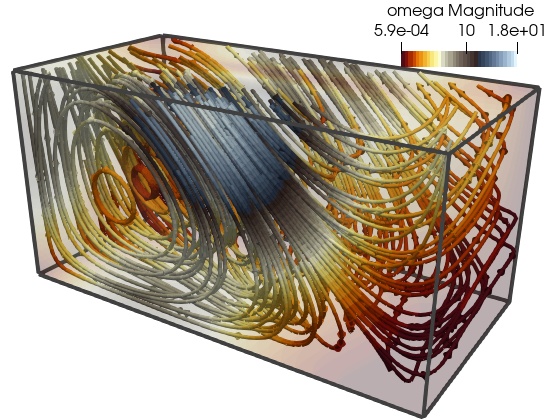}
\includegraphics[width=0.45\textwidth]{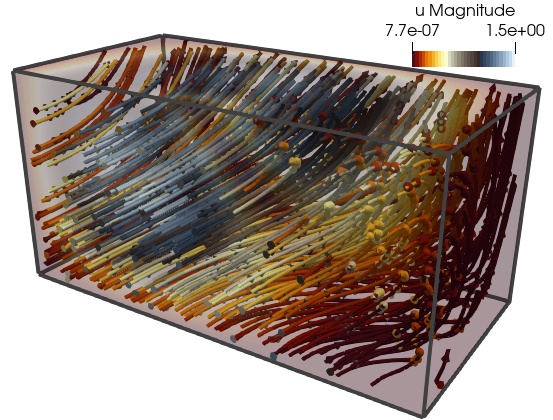}\\
\includegraphics[width=0.45\textwidth]{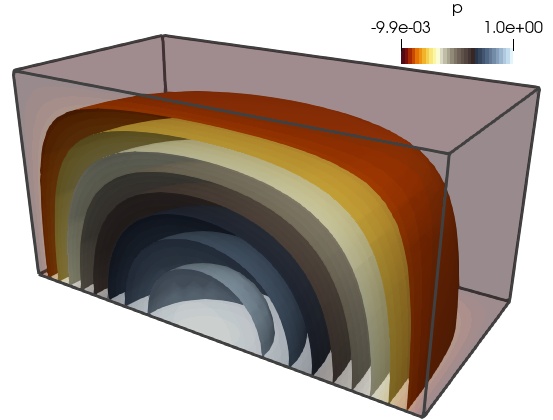}
\includegraphics[width=0.45\textwidth]{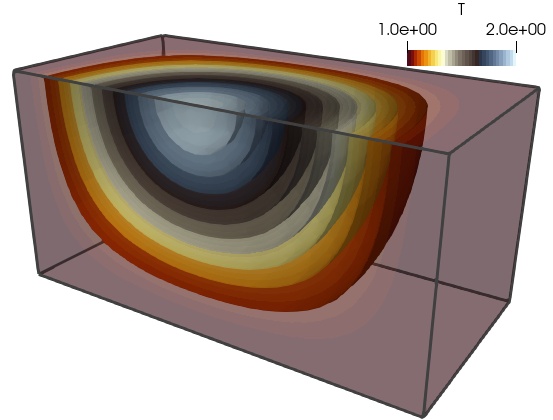}
    \caption{Accuracy tests in 2D and 3D. Approximate solutions (vorticity distribution and streamlines, velocity magnitude and streamlines, pressure profile, and temperature field) computed with the lowest-order methods.}
    \label{fig:ex01}
\end{figure}

 Over all refinements, a maximum of four iterations for the Newton--Raphson method are required to reach a tolerance (either absolute or relative) of $10^{-8}$ on the residual. 
For this particular case, and for the two lowest polynomial degrees $k=0,1$, we display the error history in Tables~\ref{table01}-\ref{table02} for the 2D and 3D case, respectively, where we observe that the method attains an asymptotic optimal convergence of $O(h^{k+1})$ as anticipated by the analysis in Section~\ref{sec:error}. The second-last column of the tables present the discrete $\ell^\infty$-norm of the divergence of the discrete velocity, confirming that the method is mass-conservative. Approximate solutions for the two cases are reported in Figure~\ref{fig:ex01}.

\subsection{Non-isothermal through-flow in a porous channel}
For our second example we model the flow of cold water through a channel with five hot cylinders. The design of the test aims at heating the liquid and observing buoyancy effects as well. The channel has length $L=1.7$ and height $H=1$ (adimensional units). The cylinders have radii of approximately 0.1 and are slightly unsymmetrically distributed. The flow enters the channel from the left segment and exits on the right end.  
For this test we  consider a time-dependent model (adding the terms $\partial_t \bu$ to the momentum balance and $\partial_t T$ to the thermal balance), and adopt a simple backward Euler time discretisation  with constant time step $\Delta t = 0.005$ and run the simulation until the final time $t=1$. We follow a similar flow configuration as in, e.g., \cite[Chapter 15]{quarteroni94} and  take the following values for the remaining parameters $\nu = 0.001$, $\gg=(0,-9.8)^{\tt t}$, $\rho = 1$, $\mu = 10^{-4}$, $K=1$, $\mu'=10^{-3}$, $\beta =c'=1$, $\alpha=0.1$. For the thermal equation we set a prescribed cold temperature on the inlet $T_{\mathrm{cold}} = 10$, hot at the cylinders' surface $T_{\mathrm{hot}}=50$, and consider zero-flux boundary conditions at the horizontal walls and at the outlet. For the flow problem we 
impose a parabolic velocity profile and a compatible vorticity on the inlet 
\[\bu_{\mathrm{in}} = (1.5\arctan(40y(1-y)),0)^{\tt t},\qquad \omega_{\mathrm{in}} = -60\sqrt{\nu}\frac{2y-1}{1600y^2(y-1)^2+1},\] 
slip velocity condition at the walls and cylinders, 
and on the outlet assume zero  pressure $p_{\mathrm{out}}=0$ (the vorticity is not prescribed at walls and cylinders). Then it suffices to add the following term in the weak form of the constitutive equation for rescaled vorticity 
\[ \sqrt{\mu'}\langle \bu_h\times\bn, \bzeta_h\rangle_{\mathrm{out}}.\]

\begin{figure}[!t]
 \begin{center}
\includegraphics[width=0.45\textwidth]{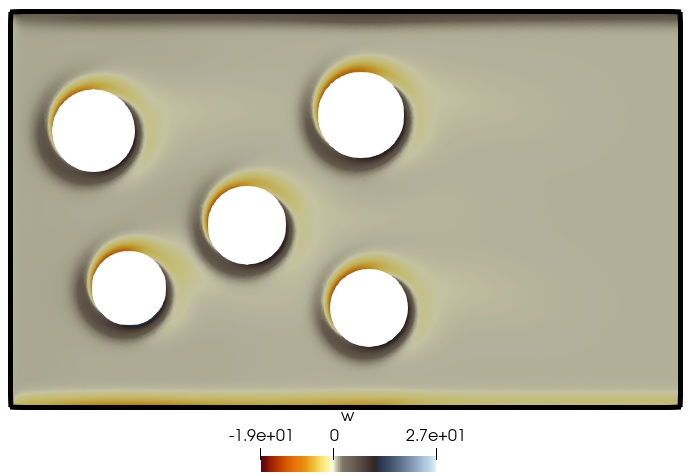}
\includegraphics[width=0.45\textwidth]{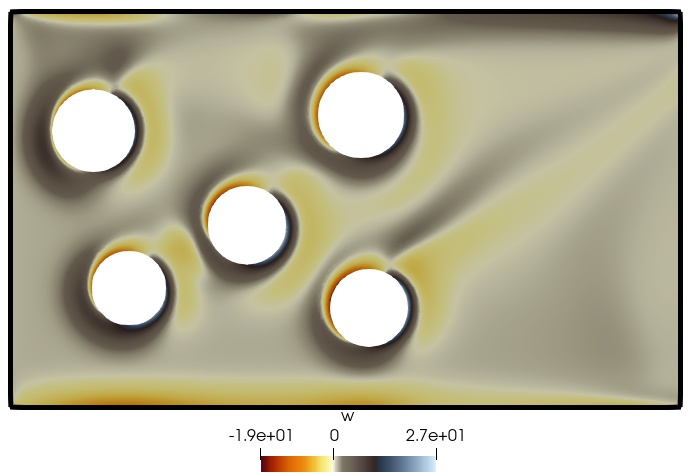}\\
\includegraphics[width=0.45\textwidth]{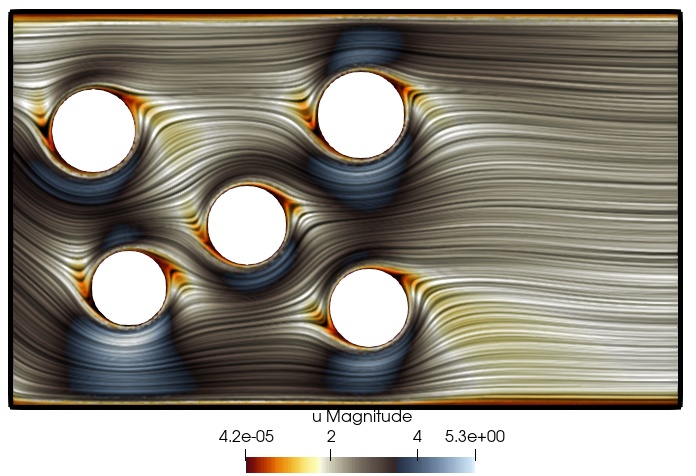}
\includegraphics[width=0.45\textwidth]{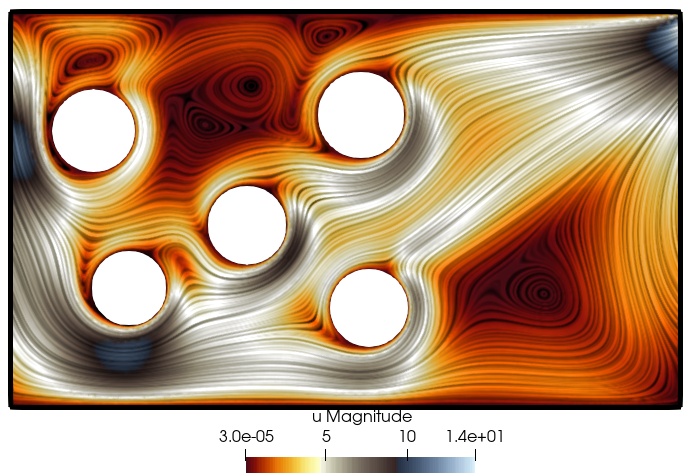}\\
\includegraphics[width=0.45\textwidth]{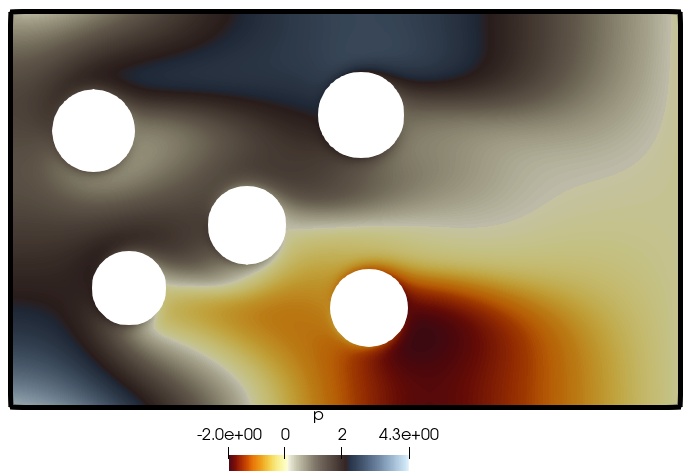}
\includegraphics[width=0.45\textwidth]{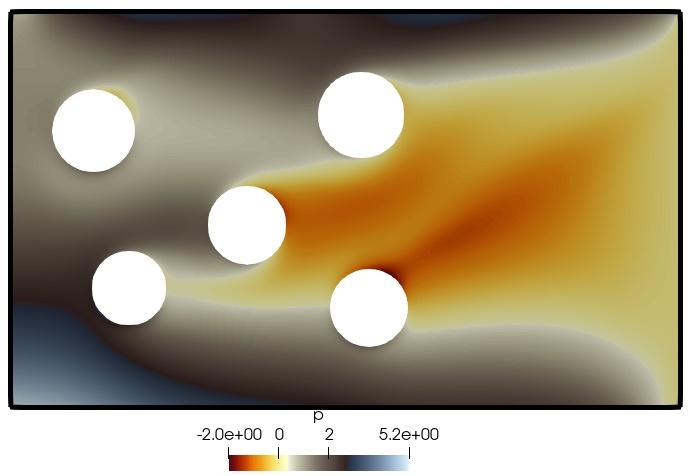}\\
\includegraphics[width=0.45\textwidth]{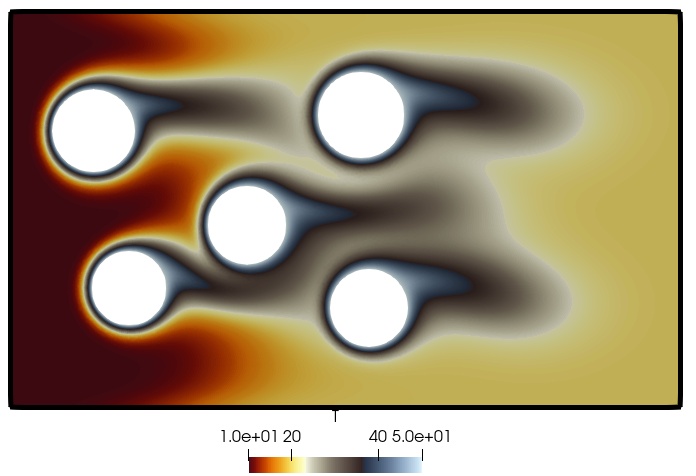}
\includegraphics[width=0.45\textwidth]{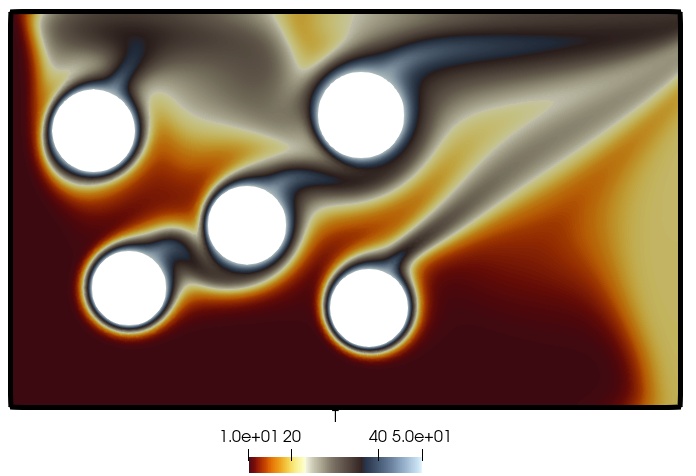}
 \end{center}
  
\vspace{-0.3cm}
\caption{Example 2. Snapshots at $t=0.2$ (left) and $t=1$ (right) for the the rescaled vorticity (top), velocity magnitude and line integral convolution (second row), pressure profile (third row), and temperature distribution (bottom) for the channel flow past five cylinders with $\mu=10^{-4}$. Here we have used the second-order scheme with $k=1$.}\label{fig:ex02}
\end{figure}

The system is able to produce the expected  phenomenon buoyancy (temperature rises towards the top wall forming channels of hot liquid), and we can also observe  vortexes behind the cylinders. The results are depicted in the different panels of Figure~\ref{fig:ex02}.  For this test we have used the second-order scheme with $k=1$.

\section{Concluding remarks}
Studying natural convection in highly permeable porous media considering vorticity and viscous dissipation has the potential to advance our understanding of complex fluid systems and improving various industrial processes. These multiphysics models include fluid dynamics, heat transfer, and porous media mechanisms. We prove that the governing equations are well-posed using the Banach fixed-point theory and perturbed saddle-point theory in Banach spaces. The discrete problem is shown to be well-posed and the analysis requires two Raviart--Thomas interpolators to match the regularity requirements of the continuous problem. We have presented numerical tests that confirm the properties of the proposed numerical methods. 
It still remains to tackle the time-dependent version of the problem, a mixed formulation for the thermal energy equation (to provide energy conservation as well), and to analyse the functional structure of the set of equations in the case of a second viscous dissipation term in the energy equation. 


\subsection*{Acknowledgement} 
The authors gratefully acknowledge the suggestions of Prof. J\"urgen Rossmann regarding the regularity of the auxiliary div-curl problem in Lemma~\ref{lem:prop}. We also thank useful comments from  Prof. Ch\'erif Amrouche and Prof. Abner Salgado. 
\bibliographystyle{siam}
\bibliography{brinkman_dissipation_references}
\end{document}